\title{Homogeneous approximation\\ Recursive observer design
and Output feedback}
\author{Vincent Andrieu%
				\thanks{V. Andrieu
				(vincent.andrieu@gmail.com) is with LAAS-CNRS, University of Toulouse, 31077 Toulouse, France.
This work has been done while he was in Electrical and Electronic Engineering Dept, Imperial College London.
        }%
        \and Laurent Praly%
				\thanks{L. Praly (Laurent.Praly@ensmp.fr) is with the Centre d'Automatique et Syst\`{e}mes,
           \'Ecole des Mines de Paris, 35 Rue Saint Honor\'{e},
           77305 Fontainebleau, France.
        }
        \and Alessandro Astolfi%
        \thanks{
        A. Astolfi (a.astolfi@ic.ac.uk) is with the Electrical and Electronic Engineering Dept,
        Imperial College London, London, SW7 2AZ, UK and with Dipartimento di Informatica Sistemi e Produzione,
        University of Rome Tor Vergata, Via del Politecnico 1, 00133 Roma, Italy.
        The work of A. Astolfi and V. Andrieu is partly supported by the Leverhulme Trust.
        }
}
\documentclass[final]{siamltex}
\usepackage{amsfonts,latexsym,amssymb,color}

\def\RR{{\mathbb R}}
\def\NN{{\mathbb N}}

\def\SR{{\mathcal S}}
\def\CR{{\mathcal C}}
\def\dr{\mathfrak d}
\def\KR{{\mathcal K}}
\def\LR{{\mathcal L}}

\def\barW{{\mathfrak W}}
\def\bare{{\vartheta}}

\def\coef{\mathfrak{c}}

\catcode`\@=11
\def\downparenfill{$\m@th\braceld\leaders\vrule\hfill\bracerd$}
\def\overparen#1{\mathop{\vbox{\ialign{##\crcr\crcr \noalign{\kern0.4ex}
\downparenfill\crcr\noalign{\kern0.4ex\nointerlineskip}
$\hfil\displaystyle{#1}\hfil$\crcr}}}\limits}
\catcode`\@=12

\newlength{\comwidth}

\def\RR{{\mathbb R}}

\def\example{\begingroup\refstepcounter{theorem}
		\par\vspace{1em}\noindent\par\noindent{\textbf{Example} \thetheorem~:~}\relax}%\ignorespaces}

\def\remark{\begingroup\refstepcounter{theorem}
		\par\vspace{1em}\noindent\par\noindent{\textbf{Remark} \thetheorem~:~}\relax}%\ignorespaces}

\def\chi{{\mathchoice
{{\mbox{$\scriptstyle \mathcal{X}$}}}
{{\mbox{$\scriptstyle \mathcal{X}$}}}
{{\mbox{$\scriptscriptstyle \mathcal{X}$}}}
{{\mbox{$\scriptscriptstyle \mathcal{X}$}}}
}}
\def\XR{\mathfrak{X}}

\begin{document}
\maketitle

\begin{abstract}
We introduce two new tools that can be useful in nonlinear observer and output feedback design.
The first one is a simple extension of the notion of homogeneous approximation to make it valid both at the origin and at infinity (homogeneity in the bi-limit).
Exploiting this extension, we give several results concerning stability and robustness for a homogeneous in the bi-limit vector field.
The second tool is a new recursive observer design procedure for a chain of integrator.
Combining these two tools we propose a new global asymptotic stabilization result by output feedback for feedback and feedforward systems.
\end{abstract}

\section{Introduction}
The problems of designing globally convergent observers and globally asymptotically stabilizing output feedback control laws for
nonlinear systems have been addressed by many authors
following different routes.
Many of these approaches exploit domination ideas and robustness of stability and/or
convergence. In view of
possibly clarifying and developing further
these techniques we introduce
two new tools.
The first one is a simple extension of the technique of homogeneous approximation to make it valid both at the origin and
at infinity.
The second tool is a new recursive observer design procedure for a chain of integrator.
Combining these two tools we propose a new global asymptotic stabilization result by output feedback for feedback and feedforward systems.

To place our contribution in perspective,
we consider the system for which we want to design a global asymptotic stabilizing output feedback~:
\begin{equation}\label{b18}
\dot x_1\;=\; x_2\quad ,\qquad
\dot x_2\;=\; u \,+\, \delta_2(x_1,x_2) \quad,\qquad y\,=\,x_1\  ,
\end{equation}
where (see notation (\ref{b52}))~:
\begin{equation}\label{b46}
\delta_2(x_1,x_2) \;=\; c_0\,x_2^q \,+\, c_\infty\,x_2^p\quad,\qquad (c_0,c_\infty)\;\in\; \RR^2
\quad,\qquad p\,>\,q\,>\,0\ .
\end{equation}
and the problem of designing a globally stabilizing output feedback controller.

In the domination's approach, the nonlinear function $\delta_2$ is not treated per se in the design but considered as a perturbation.
In this framework the output feedback controller is designed on the linear system~:
\begin{equation}\label{d1}
\dot x_1\;=\; x_2\quad ,\qquad
\dot x_2\;=\; u \quad,\qquad y\,=\,x_1\  ,
\end{equation}
and will be suitable for the nonlinear system (\ref{b18}) provided the global asymptotic stability obtained for the origin of the closed-loop system is robust to the nonlinear disturbance $\delta_2$.
For instance, the design given in \cite{Khalil-Saberi,Qian-Lin} provides a linear output feedback controller which is suitable for the nonlinear system (\ref{b18}) when $q=1$ and $c_\infty=0$.
This result has been extended recently in \cite{Qian} employing a homogeneous output feedback controller which allows to deal with $p\geq 1$ and $c_0=0$.

Homogeneity in the bi-limit and the novel recursive observer design proposed in this paper allow us to deal with the case in which $c_0\neq 0$ and $c_\infty\neq 0$, .
In this case, the function $\delta_2$ is such that~:
\begin{enumerate}
\item when $|x_2|$ is small and $q=1$, $\delta_2(x_2)$ can be approximated by $c_0\,x_2$
and the
nonlinearity %system
can be approximated by a linear
function; %system;
\item when $|x_2|$ is large, $\delta_2(x_2)$ can be approximated by
$c_\infty\, x_2^p$, hence we have a polynomial growth which can be
handled by a weighted
homogeneous controller as in \cite{Qian}.%homogeneity.
\end{enumerate}
To deal with both linear and polynomial terms we introduce a generalization of weighted homogeneity which highlights the fact that a function becomes homogeneous as the state tends to the origin or to infinity but with different weights and degrees.

The paper is organized as follows. Section \ref{b66} is devoted to general properties
related to homogeneity. After giving
the definition of homogeneous approximation we introduce homogeneous in the bi-limit functions and vector fields
(Section \ref{b67}) and
list some of their
properties (Section \ref{b55}).
Various results concerning stability and robustness
for homogeneous in the bi-limit vector fields are given in Section \ref{b68}.
In Section \ref{b69} we introduce a novel recursive observer design method for a chain of integrator.
Section \ref{b99} is devoted to the homogeneous in the bi-limit state feedback.
Finally, in Section \ref{b70}, using the previous tools we establish
new results on stabilization by output feedback.

\section*{Notation}
\begin{itemize}
\item $\RR_+$ denotes the set $[0, +\infty)$.
\item For any non-negative real number $r$
%startmodif
the function
$w\mapsto w^r$ is defined as~:
%stopmodif
\begin{equation}\label{b52}
w^r\;=\;\textrm{sign}(w)\,|w|^r\quad,\qquad \forall\; w\;\in\; \RR\  .
\end{equation}
According to this definition~:
\begin{equation}
\label{37}
\frac{dw^r}{dw}=r |w|^{r-1}\,,\; w^2 = w|w|\,,\;
(w_1>w_2\,\textrm{ and }\, r>0)\Rightarrow w_1^r>w_2^r\ .
\end{equation}
\item
%startmodif
The function $\mathfrak{H}: \RR_+^2 \to \RR_+$ is defined as
\begin{equation}\label{c41}
\mathfrak{H}(a,b)\;=\;\frac{a}{1 + a} \;[1\,+\, b]\ .
\end{equation}
%stopmodif
%
\item Given $r = (r_1, \dots, r_n)^T$ in $\RR_+^n$ and $\lambda$ in $\RR_+$,
$\lambda^r \diamond x= \left (\lambda^{r_1}\,x_1, \,\dots , \,\lambda^{r_n}\,x_n \right )^T$
is the dilation of a vector $x$ in $\RR^n$ with weight $r$.
%startmodif
Note that~:
%stopmodif
$$
\lambda_1^r\,\diamond\,( \lambda_2^r \diamond x) \;=\; (\lambda_1\,\lambda_2)^r\,\diamond\, x\ .
$$
\item Given $r = (r_1, \dots, r_n)^T$ in $(\RR_+\setminus\{0\})^n$,
$|x|_{r}= |x_1|^{\frac{1}{r_1}} \,+\,\dots\,+\, |x_n|^{\frac{1}{r_n}}$
is the homogeneous norm with weight $r$ and degree $1$.
Note that~:
$$
|\lambda^r \,\diamond\, x|_{r}\;=\; \lambda\,|x|_{r}\qquad , \quad \left|\left(\frac{1}{|x|_{r}}\right)^r \,\diamond\, x\right|_{r}\;=\;1\ .
$$
\item Given $r$ in $(\RR_+\setminus\{0\})^n$, $S_{r}=\left\{x\in\RR^n\,|\; |x|_{r}\;=\;1\right\}$ is the unity homogeneous sphere.
Note that each $x$ in $\RR^n$ can be decomposed in polar coordinates, i.e. there exist $\lambda$ in $\RR_+$ and
$\theta$ in $S_{r}$ satisfying~:
\begin{equation}\label{4}
x \;=\; \lambda^r\,\diamond\,\theta \textrm{ with }
\left\{\begin{array}{rcl}
\lambda &=& |x|_{r}\\
\theta &=& \left(\frac{1}{|x|_{r}}\right)^{r} \,\diamond\, x\ .
\end{array}\right.
\end{equation}
\end{itemize}

\section{Homogeneous approximation}
\label{b66}

\subsection{Definitions}
\label{b67}
The use of homogeneous approximations has a long history in the study
of stability of an equilibrium.
It can be traced back to Lyapunov first order approximation
theorem and has been pursued by many authors,
for example Massera \cite{Massera}, Hahn \cite{Hahn}, Hermes \cite{Hermes},
Rosier \cite{Rosier}.
Similarly this technique has been used to
investigate the behavior of the solutions of dynamical systems at
infinity, see for instance Lefschetz in \cite[IX.5]{lefschetz}
and Orsi, Praly and Mareels in \cite{orsi-praly-mareels}.
In this section, we recall the definitions of homogeneous approximation
at the origin and at infinity and restate and/or complete some
%startmodif
related results.
%stopmodif
%
\begin{definition}[Homogeneity in the $0$-limit]~
\begin{itemize}
\item
A function $\phi~:\RR^n\rightarrow\RR$ is said homogeneous in the $0$-limit with associated triple
$(r_0,d_0, \phi_0)$, where $r_0$ in $(\RR_+\setminus\{0\})^n$ is
the weight, $d_0$ in $\RR_+$ the degree and $\phi_0~:\RR^n\rightarrow\RR$
the approximating function,
if $\phi$ is continuous,  $\phi_0$ is continuous and not identically zero and, for each compact set $C$ in $\RR^n\setminus\{0\}$ and each $\varepsilon >0$, there
exists $\lambda _0$ such that~:
$$
\max_{x\,\in\,C} \left | \frac{\phi(\lambda^{r_{0}}\diamond x)}{\lambda ^{d_0}} -\phi_0(x)\right| \; \leq \; \varepsilon \quad,\qquad \forall\quad \lambda\;\in\;(0,\lambda_0]\ .
$$
\item
A vector field $f=\sum_{i=1}^n\,f_i\frac{\partial }{\partial x_i}$ is said homogeneous in the $0$-limit with associated triple
$(r_0,\dr_0,f_0)$, where $r_0$ in $(\RR_+\setminus\{0\})^n$ is the weight, $\dr_0$ in $\RR$ is the degree
and $f_0=\sum_{i=1}^n\,f_{0,i}\frac{\partial }{\partial x_i}$ the approximating vector field, if, for each $i$ in $\{1,\dots,n\}$, $\dr_0+r_{0,i}\,\geq\,0$ and the function $f_i$ is homogeneous in the $0$-limit with associated
triple $\left(r_0,\dr_0+r_{0,i},f_{0,i}\right)$.
\end{itemize}
\end{definition}

This notion of local approximation of a function or of a vector field can be found in \cite{Hermes, Rosier, Bacciotti-Rosier, Hong}.
\example \label{b76} The function $\delta_2~:\RR\rightarrow\RR$ introduced in the illustrative system  (\ref{b18})
is homogeneous in the $0$-limit with associated triple
$
\left(r_0, \, d_0, \, \delta_{2,0}\right) = (1,q, c_0\,x_2^q)
$.
Furthermore, if $q <  2$ the vector field $f(x_1,x_2)\;=\;(x_2,\delta_2(x_2))$ is homogeneous  in the $0$-limit with associated triple~:
\begin{equation}\label{b56}
\left(r_0, \, \dr_0, \, f_0\right) \;=\; \Big((2-q,1),\,q-1,\, (x_2,c_0\,x_2^q)\Big)
\ .
\end{equation}

\begin{definition}[Homogeneity in the $\infty$-limit]~
\begin{itemize}
\item
A function $\phi~:\RR^n\rightarrow\RR$ is said homogeneous in the $\infty$-limit with associated triple $(r_\infty,d_\infty, \phi_\infty)$, where $r_\infty$ in $(\RR_+\setminus\{0\})^n$ is the weight, $d_\infty$ in $\RR_+$ the degree and $\phi_\infty~:\RR^n\rightarrow\RR$ the approximating function,
if $\phi$ is continuous, $\phi_\infty$ is continuous and not identically zero and, for each compact set $C$ in  $\RR^n\setminus\{0\}$ and each $\varepsilon >0$, there
exists $\lambda _\infty$ such that~:
$$
\displaystyle
\max_{x\,\in\,C} \left | \frac{\phi(\lambda^{r_{\infty}}\diamond x)}{\lambda ^{d_\infty}} -\phi_\infty(x)\right| \; \leq \; \varepsilon \quad,\qquad \forall\quad \lambda\;\geq\;\lambda _\infty\ .
$$
\item
A vector field $f=\sum_{i=1}^n\,f_i\frac{\partial }{\partial x_i}$ is said homogeneous in the $\infty$-limit with associated
triple $(r_\infty,\dr_\infty,f_\infty)$, where $r_\infty$ in $(\RR_+\setminus\{0\})^n$ is the weight, $\dr_\infty$ in $\RR$ the degree and  $f_\infty=\sum_{i=1}^n\,f_{\infty,i}\frac{\partial }{\partial x_i}$ the approximating vector field,
if, for each $i$ in $\{1, \dots, n\}$, $\dr_\infty+r_{\infty,i}\,\geq\, 0$  and the function $f_i$ is homogeneous in the $\infty$-limit with associated triple  $\left(r_\infty,\dr_\infty+r_{\infty,i},f_{\infty,i}\right)$.
\end{itemize}
\end{definition}

\example \label{b77} The function $\delta_2~:\RR\rightarrow\RR$ given in the illustrative system (\ref{b18})
is homogeneous in the $\infty$-limit with associated triple $
\left( r_\infty,\, d_\infty,\, \delta_{2,\infty}\right)= \left(1,\, p,\, c_\infty\,x_2^p\right)
$.
Furthermore, when $p < 2$, the vector field $f(x_1,x_2)\;=\;(x_2,\delta_2(x_2))$ is homogeneous  in the $\infty$-limit with associated triple~:
\begin{equation}\label{b45}
\left( r_\infty,\, \dr_\infty,\, f_\infty\right)\;=\; \Big((2-p,1),\, p-1,\,(x_2,c_\infty\,x_2^p) \Big)\ .
\end{equation}
\begin{definition}[Homogeneity in the bi-limit]~
A function $\phi~:\RR^n\rightarrow\RR$ (or a vector field $f~:\RR^n\rightarrow\RR^n$) is said homogeneous in the bi-limit if it is homogeneous in the $0$-limit  and homogeneous in the $\infty$-limit.
\end{definition}

\remark
\label{LP12}
If a function $\phi$ (respectively a vector field $f$) is homogeneous in the bi-limit, then the approximating function $\phi_0$ or $\phi_\infty$ (resp. the approximating vector field $f_0$ or $f_\infty$) is homogeneous in the standard
sense\footnote{This is
%startmodif
proved
%stopmodif
 noting that, for all $x$ in $\RR^n$ and all $\mu$ in $\RR_+\setminus\{0\}$,$$\frac{\phi_0(\mu^{r_0}\diamond x)}{\mu^{d_0}}=\frac{1}{\mu^{d_0}}\lim_{\lambda\rightarrow 0}
\frac{\phi\left(\lambda^{r_0}\diamond(\mu^{r_0}\diamond
x)\right)}{\lambda^{d_0}}
= \lim_{\lambda\rightarrow
0}\frac{\phi\left((\lambda\mu)^{r_0}\diamond x\right)}{(\lambda\mu
)^{d_0}}=\phi_0(x)\  ,$$ and similarly for the homogeneous in the $\infty$-limit. } (with the same weight and degree).

%\noindent \textbf{Example~:}
\example As a consequence of Examples \ref{b76} and \ref{b77}, the vector field $f(x_1,x_2)\;=\;(x_2,\delta_2(x_2))$ is homogeneous in the bi-limit with associated triples given in (\ref{b56}) and (\ref{b45})
as long as
%startmodif
$0 < q <  p < 2$.
%stopmodif

\example \label{b49}
The function $x\mapsto |x|_{r_0}^{d_0} + |x|_{r_\infty}^{d_\infty}$,
 where $(d_0,d_\infty)$ are in $\RR^2_+$ and $(r_0,r_\infty)$ are in $(\RR_{+}\setminus\{0\})^{2n}$
is homogeneous in the bi-limit with associated triples
$\left(r_0, \, d_0, \, |x|_{r_0}^{d_0}\right)$ and
$\left( r_\infty,\, d_\infty,\, |x|_{r_\infty}^{d_\infty}\right)$
provided that
\begin{equation}\label{b16}
\frac{d_\infty}{r_{\infty,i}}\;>\;\frac{d_0}{r_{0,i}}\qquad,\quad\forall \;i \;\in\; \{1,\dots, n\}\ .
\end{equation}

\example \label{c23}
We recall equation (\ref{c41}) and consider two homogeneous and positive definite functions
$\phi_0~:\RR^n\rightarrow\RR_+$ and $\phi_\infty~:\RR^n\rightarrow\RR_+$
with weights $(r_0,r_\infty)$ in $(\RR_{+}\setminus\{0\})^{2n}$ and degrees
$(d_0,d_\infty)$ in $(\RR_{+}\setminus\{0\})^{2}$.
The function
$x\mapsto \mathfrak{H}(\phi_0(x),\phi_\infty (x))$
is positive definite and homogeneous  in the bi-limit with associated triples $(r_0, d_0, \phi_0)$ and $(r_\infty, d_\infty, \phi_\infty)$.
This way to construct a homogeneous in the bi-limit function from two positive definite homogenous functions is extensively used in the paper.

\subsection{Properties of homogeneous approximations}
\label{b55}
To begin with note that the weight and degree of an homogeneous in the $0$-(resp. $\infty$-)limit
function are not uniquely defined.
Indeed, if $\phi$ is homogeneous in the $0$-(resp. $\infty$-)limit with associated
triple $(r_0,d_0,\phi_0)$ (resp. $(r_\infty,d_\infty,\phi_\infty)$),
then it is also homogeneous in the $0$-(resp. $\infty$-)limit with associated triple
$(k\,r_0,k\,d_0,\phi_0)$ (resp. $(k\,r_\infty,k\,d_\infty,\phi_\infty)$) for all $k\,>\,0$.
%startmodif
(Simply change $\lambda $ in $\lambda ^k$.)
%stopmodif

It is straightforward to show that if $\phi$ and $\zeta$ are two functions homogeneous in the $0$-(resp. $\infty$-)limit, with weights $r_{\phi,0}$ and $r_{\zeta,0}$ (resp. $r_{\phi,\infty}$ and $r_{\zeta,\infty}$), degrees $d_{\phi,0}$ and $d_{\zeta,0}$ (resp. $d_{\phi,\infty}$ and $d_{\zeta,\infty}$), and approximating functions $\phi_0$  and $\zeta_0$ (resp. $\phi_\infty$ and $\zeta_\infty$) then the following holds.
\begin{enumerate}

\item[P1~:] If there exists $k$ in $\RR_+$ such that $k\,r_{\phi,0}\,=\,r_{\zeta,0}$
(resp. $k\,r_{\phi,\infty}\,=\,r_{\zeta,\infty}$) then the function $x\,\mapsto \phi(x)\,\zeta(x)$ is homogeneous  in the $0$-(resp. $\infty$-)limit  with weight $r_{\zeta,0}$, degree $k\,d_{\phi,0}+d_{\zeta,0}$ (resp. $r_{\zeta,\infty}$, $k\,d_{\phi,\infty}+d_{\zeta,\infty}$) and approximating function $x\,\mapsto \phi_0(x)\,\zeta_0(x)$ (resp. $x\,\mapsto \phi_\infty(x)\,\zeta_\infty(x)$).

\item[P2~:] If, for each $j$ in $\{1, \dots, n\}$,  $\frac{d_{\phi,0}}{r_{\phi,0,j}}\,<\,\frac{d_{\zeta,0}}{r_{\zeta,0,j}}$
(resp. $\frac{d_{\phi,\infty}}{r_{\phi,\infty,j}}\,>\,\frac{d_{\zeta,\infty}}{r_{\zeta,\infty,j}}$),
then the function $x\,\mapsto \phi(x)\,+\,\zeta(x)$ is homogeneous  in the $0$-(resp. $\infty$-)limit with degree $d_{\phi,0}$, weight $r_{\phi,0}$ (resp. $d_{\phi,\infty}$ and $r_{\phi,\infty}$) and approximating function $x\,\mapsto \phi_0(x)$ (resp. $x\,\mapsto \phi_\infty(x)$). In this case we say that the function $\phi$ \textit{dominates} the function $\zeta$ in the $0$-limit (resp. in the $\infty$-limit).

\item[P3~:] If the function $\phi_0+\zeta _0$ (resp. $\phi_\infty
+\zeta _\infty $) is not identically zero
and, for each $j$ in $\{1, \dots, n\}$,
$\frac{d_{\phi,0}}{r_{\phi,0,j}}\,=\,\frac{d_{\zeta,0}}{r_{\zeta,0,j}}$
(resp. $\frac{d_{\phi,\infty}}{r_{\phi,\infty,j}}\,=\,\frac{d_{\zeta,\infty}}{r_{\zeta,\infty,j}}$),
then the function $x\,\mapsto \phi(x)\,+\,\zeta(x)$ is homogeneous  in the $0$-(resp. $\infty$-)limit with degree $d_{\phi,0}$, weight $r_{\phi,0}$ (resp. $d_{\phi,\infty}$, $r_{\phi,\infty}$) and approximating function $x\,\mapsto \phi_0(x)\,+\,\zeta_0(x)$ (resp. $x\,\mapsto \phi_\infty(x)+\zeta_\infty(x)$).

\end{enumerate}
Some properties of the composition or inverse of functions are given in the following two propositions,
the proofs of which are given in Appendices \ref {c18} and \ref {c21}.
\begin{proposition}[Composition function]% of homogeneous in the $0$-(resp. $\infty$-)limit functions]
\label{prop1}
If $\phi~:\RR^n\rightarrow\RR$ and $\zeta~:\RR\rightarrow\RR$ are homogeneous in the $0$-(resp. $\infty$-)limit functions, with weights $r_{\phi,0}$ and $r_{\zeta,0}$ (resp. $r_{\phi,\infty}$ and $r_{\zeta,\infty}$), degrees $d_{\phi,0}>0$ and $d_{\zeta,0}\geq 0$  (resp. $d_{\phi,\infty}>0$ and $d_{\zeta,\infty}\geq 0$), and approximating functions $\phi_0$ and $\zeta_0$ (resp. $\phi_\infty$ and $\zeta_\infty$) , then $\zeta\circ \phi$ is homogeneous in the $0$-(resp. $\infty$-limit) with weight $r_{\phi,0}$ (resp. $r_{\phi, \infty}$), degree $\frac{d_{\zeta,0}\,d_{\phi,0}}{r_{\zeta,0}}$ (resp. $\frac{d_{\zeta,\infty}\,d_{\phi,\infty}}{r_{\zeta,\infty}}$), and approximating function $\zeta_0\circ \phi_0$ (resp. $\zeta_\infty\circ \phi_\infty$).
\end{proposition}

\begin{proposition}[Inverse function]% of a homogeneous in the $0$-limit function]
\label{b32}
Let $\phi~:\RR\rightarrow\RR$ be a bijective homogeneous in the $0$-(resp. $\infty$-)limit function with associated triple $\left(1,d_0,\varphi_0\,x^{d_0}\right)$ with $\varphi_0 \neq 0$ and $d_0 > 0$ (resp. $\left(1,d_\infty,\varphi_\infty\,x^{d_\infty}\right)$ with $\varphi_\infty \neq 0$ and $d_\infty > 0$).
Then the inverse function $\phi^{-1}~:\RR\rightarrow\RR$ is a homogeneous in the $0$-(resp. $\infty$-)limit function with associated triple $\left(1,\frac{1}{d_0},\left(\frac{x}{\varphi_0}\right)^\frac{1}{d_0}\right)$ (resp. $\left(1,\frac{1}{d_\infty},\left(\frac{x}{\varphi_\infty}\right)^\frac{1}{d_\infty}\right)$).
\end{proposition}

Despite the existence of well-known results concerning the derivative of a homogeneous function,
it is not possible to  say anything, in general, when dealing with homogeneity in the limit.
For example the function $$\phi(x)=x^3 + x^2 \sin(x^2) + x^3 \sin(1/x) + x^2\quad, \qquad x\,\in\,\RR\ ,$$ is homogeneous in the bi-limit with associated triples~:
$$\left(1,2,x^2\right)\quad,\qquad\left(1,3,x^3\right)\ .$$
However its derivative is neither homogeneous in the $0$-limit nor in the
$\infty $-limit.
Nevertheless the following result holds, the proof of which is elementary.
\begin{proposition}[Integral function]\label{b94}
If the function $\phi~:\RR^n \rightarrow\RR$ is homogeneous in the $0$-(resp. $\infty$-)limit with associated triple $(r_0,d_0, \phi_0)$ (resp. $(r_\infty,d_\infty, \phi_\infty)$), then the function
$\Phi_i(x) \;=\; \int_0^{x_i} \, \phi(x_1, \dots, x_{i-1}, s,x_{i+1}, \dots, x_n) \, ds$  is homogeneous in the $0$-(resp. $\infty$-)limit with associated triple $(r_0,d_0 + r_{0,i},\Phi_{i,0})$ (resp. $(r_\infty,d_\infty + r_{\infty,i},\Phi_{i,\infty})$), with $\Phi_{i,0}(x)\,=\,\int_0^{x_i}\,\phi_0(x_1,\dots,x_{i-1},s,x_{i+1}, \dots,x_n)\,ds$ (resp. $\Phi_{i,\infty}(x)\,=\,\int_0^{x_i}\,\phi_\infty(x_1,\dots,x_{i-1},s,x_{i+1}, \dots,x_n)\,ds$) .
\end{proposition}

%startmodif
%\begin{proposition}
%Results similar to those in Propositions \ref{prop1}, \ref{b32} and \ref{b94} hold for functions homogeneous in the $\infty$-limit and in the bi-limit.
%\end{proposition}
%stopmodif
%
\par\vspace{1em}
By exploiting the definition of homogeneity in the bi-limit
it is possible to establish results which are straightforward
extensions of well-known results based on the standard notion of homogeneity.
These results are given as corollaries of a key technical lemma, the proof of which is given in Appendix \ref{c20}.
\begin{lemma}[Key technical lemma]\label{3}
Let  $\eta~:\RR^n\rightarrow\RR$ and $\gamma~:\RR^n\rightarrow\RR_+$ be two functions homogeneous in the bi-limit, with weights $r_0$ and $r_\infty$, degrees $d_0$ and $d_\infty$, and  approximating functions, $\eta_0$ and $\eta_\infty$, and, $\gamma_0$ and $\gamma_\infty$ such that the following holds~:
\\
\null \hfill $
\renewcommand{\arraystretch}{1.3}
\begin{array}[t]{rcl}
\left\{\  x\in \RR^n\setminus\{0\}\: :\;  \gamma(x) = 0\  \right\} \qquad
&\subseteq& \qquad \left\{\  x\in \RR^n \: :\; \eta(x) < 0\ \right\}\ ,\\
\left\{\  x\in \RR^n\setminus\{0\} \: :\;  \gamma_0(x) = 0 \  \right\} \qquad
&\subseteq& \qquad \left\{\  x\in \RR^n \: :\; \eta_0(x) < 0\ \right\}\ ,\\
\left\{\  x\in \RR^n\setminus\{0\} \: :\;  \gamma_\infty(x) = 0 \  \right\} \qquad
&\subseteq& \qquad \left\{\  x\in \RR^n \: :\; \eta_\infty(x) < 0\ \right\}\ .
\end{array}
$\hfill \null
\par\vspace{1em}\noindent
Then there exists a real number $c^*$ such that, for all $c\,\geq\,c^*$, and for all $x$ in $\RR^n\setminus\{0\}$~:
\begin{equation}\label{b53}
\eta(x) - c\,\gamma(x) <0\;,\quad \eta_0(x) - c\,\gamma_0(x) <0\;,\quad\eta_\infty(x) - c\,\gamma_\infty(x) <0\ .
\end{equation}
\end{lemma}
\example
\label{lp9}
%startmodif
To illustrate the importance of this Lemma, consider,
for $(x_1,x_2)$ in $\RR^2$, the functions
$$\eta
(x_1,x_2) = x_1\,  x_2 \,-\, |x_1|^{\frac{r_1+r_2}{r_1}}\quad ,\qquad
\gamma (x_1,x_2) = |x_2|^{\frac{r_1+r_2}{r_2}}
\  ,
$$
with $r_1>0$ and, $r_2 >0$,
They are homogeneous in the standard sense and therefore in the
bi-limit, with same weight
$r=(r_1,r_2)$ and same degree $d=r_1+r_2$.% and $\gamma $ is positive.
Furthermore the function $\gamma$ takes positive values and for all $(x_1,x_2)$ in $\{(x_1,x_2)\in\RR^2\setminus\{0\}\::\;\gamma(x_1,x_2)=0\}$ we have
$$
\eta(x_1, x_2)\;=\; - |x_1|^{\frac{r_1+r_2}{r_1}} \,<\,0\ .
$$
So Lemma \ref{3} yields the existence of a positive real number $c^*$, such that for all $c\geq c*$, we have~:
\begin{equation}
x_1 \,x_2 \,-\, |x_1|^{\frac{r_1+r_2}{r_1}}\,-\,c\,
|x_2|^{\frac{r_1+r_2}{r_2}} < 0
\qquad \forall (x_1,x_2)\in\RR^2\setminus \{0\}
\  .
\end{equation}
This is a generalization of the procedure known as the completion of the squares in which however the constant $c_1^*$ is not specified.

\begin{corollary}\label{b25}
Let  $\phi~:\RR^n\rightarrow\RR$ and $\zeta~:\RR^n\rightarrow\RR_+$ be two homogeneous in the bi-limit functions with the same weights $r_0$ and $r_\infty$, degrees
$d_{\phi,0}$, $d_{\phi,\infty}$ and $d_{\zeta ,0}$, $d_{\zeta ,\infty}$, and approximating functions $\eta_0$, $\phi_\infty$ and $\zeta _0$, $\zeta _\infty$.
If the degrees satisfy $d_{\phi,0}\geq d_{\zeta ,0}$ and $d_{\phi,\infty}\leq d_{\zeta ,\infty}$ and
the functions $\zeta $, $\zeta _0$ and $\zeta _\infty$ are positive definite then there exists a positive real number $c$
satisfying~:
$$
\phi(x)\;\leq\; c\,\zeta (x) \quad,\; \forall\;x\;\in\;\RR^n\ .
$$
\end{corollary}
\begin{proof}
Consider the two functions
$$
\eta(x)\;:=\; \phi(x)+\zeta(x)\quad,\qquad\gamma(x)\;:=\;\zeta(x)\ .
$$
By property P2 (or  P3\footnote{%
% BEGIN FOOTNOTE
%startmodif
If $\phi_0(x)+\zeta _0(x)=0$, respectively $\phi_\infty (x)+\zeta
_\infty (x)=0$, the proof can be completed replacing $\zeta $ with $2\zeta $.
%stopmodif
% END FOONOTE
}) in Subsection \ref{b55}, they are homogeneous in the bi-limit with degrees $d_{\zeta,0}$ and $d_{\zeta ,\infty}$.
The function $\gamma$ and its homogeneous approximations being positive definite, all assumptions of Lemma \ref{3} are  satisfied.
Therefore there exists a positive real number $c$ such that~:
%startmodif
$$
c\,\gamma(x)\;>\;\eta(x)\;>\; \phi(x)
\qquad  \forall x \in \RR^n\setminus\{0\}\ .
$$
%stopmodif
Finally,  by continuity of the functions $\phi$ and $\zeta$ at zero, we can
obtain the claim.
\end{proof}

\subsection{Stability and homogeneous approximation}
\label{b68}

%startmodif
A very basic property of asymptotic stability is its robustness. This
fact was already known to Lyapunov who proposed
his second method,  (local) asymptotic stability of an
equilibrium is established
by looking at the first order approximation of the system.
%stopmodif
The case of local homogeneous approximations of higher degree has
been investigated by Massera \cite{Massera}, Hermes \cite{Hermes} and
Rosier \cite{Rosier}.
\begin{proposition}[\cite{Rosier}]\label{b71}
Consider a homogeneous in the $0$-limit vector field $f~:\RR^n\rightarrow\RR^n$ with associated triple $(r_0,\dr_0,f_0)$.
If the origin of the system~:
$$
\dot x \;=\; f_0(x)
$$
is locally asymptotically stable then the origin of
$$
\dot x \;=\; f(x)
$$
is locally asymptotically stable.
\end{proposition}

Consequently, a natural strategy to ensure local asymptotic stability
of an equilibrium of a system is to design a stabilizing homogeneous control law for the homogeneous approximation in the
$0$-limit (see \cite{Hermes,KawskiSCL,Coron-Praly} for instance).
\example Consider the system (\ref{b18}) with $q\,=\,1$ and $p\,>\,q$, and the linear control law~:
$$
u \;=\; - (c_0+1)\,x_2 \,- \,x_1\ .
$$
The closed loop vector field is homogeneous in the $0$-limit with degree $\dr_0=0$, weight $(1,1)$ (i.e. we are in the linear
case) and associated vector field
$f_0(x_1,x_2) = \left(x_2,-x_1-x_2\right)^T$.
Selecting the Lyapunov function of degree two~:
$$
V_0(x_1,x_2) \;=\; \frac{1}{2}\,|x_1|^2\;+\; \frac{1}{2}\,\left|x_2\,+\,x_1\right|^2\ ,
$$
yields~:
$$
\frac{\partial V_0}{\partial x}(x)\,f_0(x)\;=\;
%-|x_1|^\frac{p+1}{2-p} \;+\;(x_2\,+\,x_1^\frac{1}{2-p}) \left( u\,+\,x_1^\frac{p-1}{2-p}x_2\,+\,x_1^\frac{p}{2-p}\right)
-|x_1|^2 \;-\;\left|x_2\,+\,x_1\right|^{2}\ .
$$
It follows, from Lyapunov second method, that the control law locally asymptotically stabilizes the equilibrium of the system.
Furthermore, local asymptotic stability is
preserved in the presence of any perturbation which does not change the approximating
homogeneous function, i.e., in the presence of perturbations which are dominated by the
linear part
(see Point P2 in Section \ref{b55}).

In the context of homogeneity in the $\infty$-limit,  we have the following result.
\begin{proposition}
\label{prop2}
Consider a homogeneous in the $\infty$-limit vector field $f~:\RR^n\rightarrow\RR^n$ with associated triple $(r_\infty,\dr_\infty,f_\infty)$.
If the origin of the system~:
$$
\dot x \;=\; f_\infty(x)\ ,
$$
is globally asymptotically stable then there exists an invariant compact subset of $\RR^n$, denoted
$\CR_\infty$, which is globally asymptotically stable\footnote{
% BEGIN FFOOTNOTE
%startmodif
See \cite{WesleyWilson} for the definition of
global asymptotical stability for invariant compact sets.
%stopmodif
% END FOONOTE
} for the system~:
$$
\dot x \;=\; f(x)\ .
$$
\end{proposition}
The proof of the proposition is given in Appendix \ref{c14}.

As in the case of homogeneity in the $0$-limit, this property can be
used to design a feedback ensuring boundedness of  solutions.

\example Consider the system (\ref{b18}) with  $0\,<\,q\,<\,p   \,<\, 2$ and
 the control law~:
\begin{equation}\label{b72}
u \;=\; - \frac{1}{2-p}\,x_1^\frac{p-1}{2-p}x_2\,-\,x_1^\frac{p}{2-p}\,-\,c_\infty\,x_2^p-\left(x_2\,+\,x_1^\frac{1}{2-p}\right)^p\ .
\end{equation}
This control law is such that the closed loop vector field is homogeneous in the $\infty$-limit with degree $\dr_\infty=p-1$, weight $(2-p,1)$ and associated vector field
$f_\infty(x_1,x_2) = \left(x_2,- \frac{1}{2-p}\,x_1^\frac{p-1}{2-p}x_2\,-\,x_1^\frac{p}{2-p}-\left(x_2\,+\,x_1^\frac{1}{2-p}\right)^p\right)^T$.
For the homogeneous Lyapunov function of degree two~:
$$
V_\infty(x_1,x_2) \;=\; \frac{2-p}{2}\,|x_1|^\frac{2}{2-p}\;+\; \frac{1}{2}\,\left|x_2\,+\,x_1^\frac{1}{2-p}\right|^2\ ,
$$
we get~:
$$
\frac{\partial V_\infty}{\partial x}(x)\,f_\infty(x)\;=\;
%-|x_1|^\frac{p+1}{2-p} \;+\;(x_2\,+\,x_1^\frac{1}{2-p}) \left( u\,+\,x_1^\frac{p-1}{2-p}x_2\,+\,x_1^\frac{p}{2-p}\right)
-|x_1|^\frac{p+1}{2-p} \;-\;\left|x_2\,+\,x_1^\frac{1}{2-p}\right|^{p+1}\ .
$$
It follows that the control law (\ref{b72}) guarantees boundedness of
the solutions of the closed loop system.
Furthermore, boundedness of solutions is preserved
in the presence of any perturbation which does not change the approximating
homogeneous function in the $\infty$-limit, i.e.
in the presence of perturbations which are negligible with respect to the
dominant homogeneous part (see Point P2 in Section \ref{b55}).

The key step in the proof of Propositions \ref{b71} and \ref{prop2} is the converse Lyapunov theorem given by Rosier in \cite{Rosier}.
This result can also be extended to the
%startmodif
case of homogeneity in the bi-limit.
%stopmodif
%
\begin{theorem}[Homogeneous in the bi-limit Lyapunov functions]
\label{b26}
Consider a homogeneous in the bi-limit vector field $f~:\RR^n\rightarrow\RR^n$, with associated triples $(r_\infty,\dr_\infty,f_\infty)$ and $(r_0,\dr_0,f_0)$ such that the origins of the systems~:
\begin{equation}\label{c19}
\dot x = f(x)\qquad,\quad \dot x = f_\infty(x)\qquad,\quad \dot x = f_0(x)
\end{equation}
are globally asymptotically stable equilibria.
Let $d_{V_\infty}$ and $d_{V_0}$ be real numbers such that
$d_{V_\infty}>\max_{1\leq i\leq n}r_{\infty,i}$ and $d_{V_0}>\max_{1\leq i\leq n}r_{0,i}$.
Then there exists a $C^1$, positive definite and proper function $V~:\RR^n\rightarrow\RR_+$
such that, for each $i$ in $\{1, \dots, n\}$, the functions  $x \mapsto\frac{\partial V}{\partial x_i}$ is
homogeneous in the bi-limit with associated triples
$\left(r_0,d_{V_0}-r_{0,i},\frac{\partial V_0}{\partial x_i}\right)$ and
$\left(r_\infty,d_{V_\infty}-r_{\infty,i},\frac{\partial V_\infty}{\partial x_i}\right)$ and
the function $x\mapsto \frac{\partial V}{\partial x}(x)\,f(x)$,
$x\mapsto \frac{\partial V_0}{\partial x}(x)\,f_0(x) $
and
$x\mapsto \frac{\partial V_\infty}{\partial x}(x)\,f_\infty(x) $
are negative definite.
\end{theorem}

The proof is given in Appendix \ref{c13}.
A direct consequence of this result is an Input-to-State Stability
(ISS) property with respect to disturbances (see \cite{Sontag}).
To illustrate this property, consider the system with exogenous disturbance $\delta=(\delta_1, \dots, \delta_m)$ in $\RR^m$~:
\begin{equation}\label{c26}
\dot x = f(x,\delta)\ ,
\end{equation}
with $f~:\RR^n\times\RR^m$ a continuous vector field homogeneous in the bi-limit with associated
triples $(\dr_0, (r_0,\mathfrak{r}_{0}), f_{0})$ and
$(\dr_\infty,(r_\infty, \mathfrak{r}_{\infty}), f_{\infty})$
% \footnote{the functions $f_i~:\RR^n\times\RR^m\rightarrow\RR$ are homogeneous
% in the bi-limit with associated triples $(d_0+r_{0,i}, (r_0,\mathfrak{r}_{0}), f_{i,0})$ and
% $(d_\infty+r_{\infty,i},(r_\infty, \mathfrak{r}_\infty), f_{i,\infty})$
%}
where $\mathfrak{r}_{0}$ and $\mathfrak{r}_{\infty}$ in $(\RR_+\setminus\{0\})^m$ are the weights
associated to the disturbance $\delta$.
\begin{corollary}[ISS Property]\label{b51}
If  the origins of the systems~:
$$
\dot x = f(x,0)\quad, \qquad \dot x = f_0(x,0)\quad, \qquad \dot x = f_\infty(x,0)
$$
are globally asymptotically stable equilibria,
then under the hypotheses of Theorem \ref{b26}
the function $V$ given by Theorem \ref{b26} satisfies\footnote{The function $\mathfrak{H}$ is
defined in (\ref{c41}). } for all
$\delta=(\delta_1, \dots, \delta_m)$ in $\RR^m$ and $x$ in $\RR^n$~:
\\[1em]$
\displaystyle\frac{\partial V}{\partial x}(x)\,f(x,\delta)
\;\leq\; -c_V\;
\mathfrak{H}\left(V(x)^\frac{d_{V_0}+d_0}{d_{V_0}},V(x)^\frac{d_{V_\infty}+d_\infty}{d_{V_\infty}}\right)
$\hfill\null\\
\refstepcounter{equation}$(\theequation)$\hfill\label{b98} $\displaystyle  \;+\;c_\delta\sum_{j=1}^m\,
\mathfrak{H}\left(|\delta_j|^\frac{d_{V_0}+d_0}{ \mathfrak{r}_{0,j}},|\delta_j|^\frac{d_{V_\infty}+d_\infty}{\mathfrak{r}_{\infty,j}}\right)\ ,
$\\[0.5em]
where $c_V$ and $c_\delta$ are positive real numbers.\\[1em]
\end{corollary}
In other words, system (\ref{c26}) with $\delta$ as input is ISS.
The proof of this corollary is given in Appendix \ref{c03}.

Finally, we have also the following small-gain result for homogeneous in the bi-limit vector fields.
\begin{corollary}[Small-Gain]\label{b36}
Under the hypotheses of Corollary \ref{b51}, there exists a real number
$c_G>0$ such that,
%startmodif
for each class $\mathcal{K}$ function $\gamma _z$ and
$\mathcal{KL}$ function $\beta _\delta $, there exists a class
$\mathcal{KL}$ function $\beta _x$ such that, for each function
$t\in [0,T)\mapsto (x(t),\delta (t),z(t))$, $T\leq +\infty $,
with $x$ $C^1$ and $\delta $ and $z$
continuous, which satisfies, both (\ref{c26}) on $[0,T)$
and,  for all $0\leq s\leq t\leq T$,
\begin{eqnarray}
\label{LP11}
|z(t)|&\leq &
\max\left\{
\beta _\delta \Big(|z(s)|, t-s\Big)\,  ,\,
\sup_{s\leq \kappa \leq t}
\gamma _z(|x(\kappa)|)\right\}\  ,
\\\label{b96}
\null \qquad |\delta _i(t)|
&\leq& \max\left\{
\beta _\delta \Big(|z(s)|, t-s\Big)\,  ,\,  c _G\,
\sup_{s\leq \kappa \leq t}\left\{
\mathfrak{H}\left(|x(\kappa)|_{r_0}^{\mathfrak{r}_{0,i}},|x(\kappa)|_{r_\infty}^{\mathfrak{r}_{\infty,i}}
\right)
\right\}\right\}\  ,
\end{eqnarray}
we have
\begin{equation}
|x(t)|\leq \beta _x(|(x(s),z(s))|,t-s)\qquad 0\leq s\leq t\leq T\ .
\end{equation}
\end{corollary}

The proof is given in Appendix \ref{c04}.
\example
\label{c53}
An interesting case which can be dealt with by Corollary \ref{b36}
is when the $\delta _i$'s are outputs of auxiliary systems with state $z _i$ in $\RR^{n_i}$, i.e~:
\begin{equation}
\label{c15}
\delta _i(t)\,:=\,\delta _i(z_i(t),x(t)) \quad,\qquad \dot z _i \;=\; g _i(z _i,x)\ .
\end{equation}
It can be
checked that the bounds (\ref{b96}) and (\ref{LP11})
are satisfied
by all the solutions of (\ref{c26}) and (\ref{c15}) if there exist
positive definite and radially unbounded functions
% \footnote{
% %
%startmodif
% There is no need that the function $Z_i$ to be proper and positive definite.
% It is sufficient to have
%stopmodif
% the solution of the system (\ref{c15}) to be complete in
% positive time
% and the existence of a class $\KR_\infty$ functions $\omega _4$ such that~:
% $$
% Z _i(z _i)\;\leq\; \omega _4(|\delta _i(z _i)|)
% $$
% Nevertheless, if the function $Z _i$ positive definite and proper, it ensures that the
% $z_i$ dynamics converge to zero with the stabilizing output feedback.}
%
$Z _i~:\RR^{n_i}\rightarrow \RR_+$, class $\KR$ functions $\omega_1$,
$\omega _2$
and $\omega _3$,
a positive real number $\epsilon$ in $(0,1)$
such that for all $x$ in $\RR^n$, for all $i$ in $\{1, \dots, m\}$ and $z_i$ in $\RR^{n_i}$, we have~:
\begin{eqnarray*}
&\displaystyle
|\delta _i(z _i,x)|\,\leq\,\omega _1(x)+\omega _2(Z _i(z_i))
\quad,\qquad \frac{\partial Z _i}{\partial z _i}(z _i)\,g _i(z _i,x)
\;\leq\; -Z _i(z _i) \;+\; \omega _3(|x|)
%\quad,\qquad \omega _3(|z _i|)\;\leq\;Z _i(z _i)\;\leq\; \omega _4(|z _i|)
\  ,
\\
&\displaystyle
\omega _1(x) + \omega _2\left([1+\epsilon]\,\omega _3(|x|)\right)\;\leq\;
% c_0\,  \sum_{j=1}^i |x_j|^\frac{1-\dr_0(n-i-1)}{1-\dr_0(n-j)}
% \;+\; c_\infty\,\sum_{j=1}^i |x_j|^\frac{1-\dr_\infty(n-i-1)}{1-\dr_\infty(n-j)}
% \ .
c_G\mathfrak H\left(|x|_{r_0}^{\mathfrak{r}_{0,i}},|x|_{r_\infty}^{\mathfrak{r}_{\infty,i}}
\right)
\ .
\end{eqnarray*}
\par\vspace{1em}
Another important result exploiting Theorem \ref{b26}
deals with finite time convergence of solutions to the origin when this is a globally asymptotically stable equilibrium (see \cite{Bhat-Bernstein-2000}).
It is well known that when the origin of the homogeneous approximation in the $0$-limit is globally asymptotically stable and with a strictly negative degree then solutions converge to the origin in finite time (see \cite{Bhat-Bernstein}).
We extend this result by showing that if, furthermore the origin of the homogeneous approximation  in the $\infty$-limit is globally asymptotically stable with strictly positive degree then the convergence time  doesn't depend on the initial condition.
This is expressed by the following corollary.
\begin{corollary}[Uniform and Finite Time Convergence]\label{b64}
Under the hypotheses of Theorem \ref{b26}, if we have $\dr_\infty>0 > \dr_0$,
then all solutions of the system $\dot x=f(x)$ converge in finite time to the origin, uniformly in the initial condition.
\end{corollary}

The proof is given in Appendix \ref{c05}.
\section{Recursive observer design for a chain of integrators}
\label{b69}

The notion of homogeneity in the bi-limit is instrumental to introduce a new observer design method.
Throughout this section we consider a chain of integrators, with state
$\XR_n\,=\,(\chi_1,\dots,\chi_n)$ in $\RR^n$, namely~:
\begin{equation}
\label{5}
\null \qquad \dot \chi_1 \,=\, \chi_2\ ,\
\dots\ , \
\dot \chi_n \,=\, u\quad\textrm{or in compact form} \qquad
\dot \XR_n \;=\;\SR_n\,  \XR_n\,+\, B_n\,  u\ ,
\end{equation}
where $\SR_n$ is the shift matrix
of order $n$, i.e. $\SR_n\,  \XR_n\;=\;
\left(\chi_{2},\dots,\chi_n,0\right)^T$ and $B_n=(0,\dots,0,1)^T$.
By selecting arbitrary vector field degrees $\dr_0$ and $\dr_\infty$ in
$\left(-1,\frac{1}{n-1}\right)$, we see that, to possibly obtain
homogeneity in the bi-limit of the associated vector field, we must choose the
weights $r_0=(r_{0,1},\dots, r_{0,n})$ and  $r_\infty=(r_{\infty,1},\dots, r_{\infty,n})$
as~:
\begin{equation}\label{b60}
\renewcommand{\arraystretch}{1.3}
\begin{array}{ccccccc}
r_{0,n}  &=& 1\  ,\quad r_{0,i}&=& r_{0,i+1}\,-\,\dr_0&=& 1 \,-\, \dr_0 \,(n-i)\ ,\\
r_{\infty,n} &=&  1\  ,\quad r_{\infty,i} &=&  r_{\infty,i+1}\,-\,\dr_\infty &=&  1 \,-\, \dr_\infty \,(n-i)\ .
\end{array}
\end{equation}
The goal of this section is to introduce a global homogeneous in the
bi-limit observer for the system (\ref{5}).
This design follows a recursive method, which constitutes one
of the main contribution of this paper.

The idea of designing an observer recursively starting from $\chi_n$ and going backwards towards $\chi_1$ is not new.
It can be found for instance in
\cite{Qian-Lin-2006,Qian,Praly-Jiang2,shim-seo,Yang-Lin}) and \cite[Lemma
6.2.1]{Gauthier-Kupka}).
Nevertheless, the procedure we propose is new, and extends
to the homogeneous in the bi-limit case the results in \cite[Lemmas 1 and 2]{Praly-Jiang2}.

Also, as opposed to what is proposed in \cite{Qian-Lin-2006,Qian}\footnote{Note the term $x_i$ in (3.15) of \cite{Qian-Lin-2006} for instance.}, this observer
is an exact observer (with any input $u$) for a chain of integrators.
The observer is given by the system\footnote{%
% BEGIN FOOTNOTE
%startmodif
To simplify the presentation, we use the compact notation $K_1(\hat\chi_1-\chi_1)$
for what should be $K_1(\hat\chi_1-\chi_1,0\ldots,0)$.
%stopmodif
% END FOONOTE
}~:
\begin{equation}
\label{23}
\dot {\hat \XR}_n \;=\; \SR_n\,  {\hat \XR}_n\,+\, B_n\,  u\ \,+\, K_1(\hat \chi_1 - \chi_1)
\end{equation}
with state $\hat \XR_n\,=\, (\hat \chi_1, \dots, \hat \chi_n)$, and
where $K_1~:\RR^n\rightarrow\RR^n$ is a homogeneous in the bi-limit vector field
with weights $r_0$ and $r_\infty$, and degrees $\dr_0$ and $\dr_\infty$.
The output injection

 vector field

$K_1$ has to be selected such that the origin is a globally asymptotically stable equilibrium
for the system ~:
\begin{equation}\label{c44}
\dot E_1 \;=\; \SR_n\,  E_1 \,+\, K_1(e_1)\quad, \qquad E_1\;=\;(e_1,\dots e_n)^T \ ,
\end{equation}
and also for its homogeneous approximations.
The construction of $K_1$ is performed via a recursive procedure
%startmodif
whose induction argument is as follows.
%stopmodif

Consider the system on $\RR^{n-i}$ given by~:
\begin{equation}\label{b34}
\dot E_{i+1} \;=\; \SR_{n-i}\,  E_{i+1} + K_{i+1}(e_{i+1})\quad,
\qquad E_{i+1}\,=\,(e_{i+1}, \dots, e_n)^T\ ,
\end{equation}
with $\SR_{n-i}$ the shift matrix of order $n-i$, i.e.
$\SR_{n-i}\,E_{i+1} = \left(e_{i+2},\dots,e_n,0\right)^T$,
and $K_{i+1}~:\RR^{n-i}\rightarrow\RR^{n-i}$ a  homogeneous in the bi-limit vector field,
whose associated
triples are $\left((r_{0,i+1},\dots, r_{0,n}), \dr_0, K_{i+1,0} \right)$  and
$\left((r_{\infty,i+1},\dots, r_{\infty,n}), \dr_\infty, K_{i+1,\infty} \right)$.

\begin{theorem}[Homogeneous in the bi-limit observer design]
\label{theo2}
Consider the system (\ref{b34}) and its homogeneous approximation at infinity and around the origin~:
$$
\dot E_{i+1} = \SR_{n-i}\,  E_{i+1} + K_{i+1,0}(e_{i+1})\;,\quad
\dot E_{i+1} = \SR_{n-i}\,  E_{i+1} + K_{i+1,\infty}(e_{i+1})\ .
$$
Suppose the origin is a globally asymptotically stable equilibrium for these
systems.
Then there exists a homogeneous in the bi-limit vector field
%startmodif
$K_i~:\RR^{n-i+1}\rightarrow\RR^{n-i+1}$,
%stopmodif
with associated triples $\left((r_{0,i},\dots, r_{0,n}), \dr_0, K_{i,0} \right)$ and $\left((r_{\infty,i},\dots, r_{\infty,n}), \dr_\infty, K_{i,\infty} \right)$,
such that the origin is a globally asymptotically stable equilibrium for the systems~:
\begin{eqnarray}
\nonumber
\dot E_{i} &=& \SR_{n-i+1}\,  E_{i} + K_{i}(e_{i})\  ,
\\\label{LP1}
\dot E_{i} &=&
\SR_{n-i+1}\,  E_{i} + K_{i,0}(e_{i})\quad, \qquad E_{i}\,=\,(e_{i}, \dots, e_n)^T\ ,
\\\nonumber
\dot E_{i} &=& \SR_{n-i+1}\,  E_{i} + K_{i,\infty}(e_{i})
\  .
\end{eqnarray}
\end{theorem}

\begin{proof}
We prove this result in two steps.
First we define an homogeneous in the bi-limit Lyapunov function.
Then we construct the vector field $K_i$,  depending on a parameter
$\ell$ which, if sufficiently large, renders negative definite the
derivative of this Lyapunov function along the
solutions of the system.
\par\vspace{1em}\noindent
\textbf{1. Definition of the Lyapunov function~:}
Let $d_{W_0}$ and $d_{W_\infty}$ be positive real numbers satisfying~: \\[0.5em]
\refstepcounter{equation}$(\theequation)$\label{c29}\hfill$d_{W_0} > 2\,\max_{1 \leq j\leq\,n}r_{0,j}+\dr_0$ \quad , \qquad$d_{W_\infty} > 2\,\max_{1 \leq j\leq\,n}r_{\infty,j}+\dr_\infty\ ,$  \hfill\null\\[0.5em]
and
\begin{equation}\label{c28}
\frac{d_{W_\infty}}{r_{\infty,i}}\; \geq\;  \frac{d_{W_0}}{r_{0,i}} \  .
\end{equation}
The selection (\ref{b60}) implies $r_{0,j}+\dr_{0}>0$ and $r_{\infty,j}+\dr_{\infty}>0$ for each $j$ in $\{1, \dots, n\}$.
Hence,
$$d_{W_0} > \max_{1 \leq j\leq\,n}r_{0,j}
\quad, \qquad d_{W_\infty} >  \max_{1 \leq j\leq\,n}r_{\infty,j} \ ,$$
and we can invoke Theorem \ref{b26} for the system (\ref{c44}) and its homogeneous approximations
given in (\ref{b34}). This implies that there exists a $C^1$, positive definite and proper function
$W_{i+1}~:\RR^{n-i}\rightarrow\RR_+$ such that, for each $j$ in $\{i+1, \dots, n\}$,
the function  $\frac{\partial W_{i+1}}{\partial e_j}$ is homogeneous in the
bi-limit with associated triples
$\left( (r_{0,i+1},\dots,r_{0,n}), d_{W_0}-r_{0,j}, \frac{\partial
W_{i+1,0}}{\partial e_j} \right)$
and
$\Big( (r_{\infty,i+1},\dots,r_{\infty,n}), d_{W_\infty}-r_{\infty,j},$ $ \frac{\partial W_{i+1,\infty}}{\partial e_j} \Big)$.
Moreover, for all $E_{i+1} \in \RR^{n-i}\setminus\{0\}$, we have~:
\begin{eqnarray}
\nonumber
\frac{\partial W_{i+1}}{\partial
E_{i+1}}(E_{i+1})\,\left( \SR_{n-i}\,  E_{i+1} + K_{i+1}(e_{i+1}) \right) &<& 0\ ,
\\\label{b41}
\frac{\partial W_{i+1,0}}{\partial
E_{i+1}}(E_{i+1})\,\left( \SR_{n-i}\,  E_{i+1} + K_{i+1,0}(e_{i+1}) \right)&<& 0\  ,
\\\nonumber
\null \qquad \quad \frac{\partial W_{i+1,\infty}}{\partial
E_{i+1}}(E_{i+1})\,\left( \SR_{n-i}\,  E_{i+1} + K_{i+1,\infty}(e_{i+1}) \right)&<& 0
\  .
\end{eqnarray}
Consider the function $q_i:\RR\to\RR$ defined as~:
\begin{equation}\label{c31}
q_i(s) \;=\;\left\{
\begin{array}{ll}
\frac{r_{0,i}}{r_{0,i}+\dr_{0}}\;s^\frac{r_{0,i}+\dr_{0}}{r_{0,i}}\ , &\quad |s|\leq 1\ ,\\
\frac{r_{\infty,i}}{r_{\infty,i}+\dr_{\infty}}s^\frac{r_{\infty,i}+\dr_{\infty}}{r_{\infty,i}} + \frac{r_{0,i}}{r_{0,i}+\dr_{0}} -  \frac{r_{\infty,i}}{r_{\infty,i}+\dr_{\infty}}\ ,&\quad |s|\geq 1\ .
\end{array}
 \right.
\end{equation}
Since we have
$0<r_{0,i}+\dr_{0}$ and $0<r_{\infty,i}+\dr_{\infty}$ ,
this function is well defined and continuous on $\RR$, strictly increasing and onto, and $C^1$ on $\RR\setminus\{0\}$.
Furthermore, it is by construction homogeneous in the bi-limit with approximating continuous functions
%startmodif
$\frac{r_{0,i}}{r_{0,i}+\dr_{0}}s^\frac{r_{0,i}+\dr_{0}}{r_{0,i}}$
and
$\frac{r_{\infty,i}+\dr_{\infty}}{r_{\infty,i}}s^\frac{r_{\infty,i}+\dr_{\infty}}{r_{\infty,i}}$.
%stopmodif
The inverse function  $q_i^{-1}$ of $q_i$ is defined as~:
$$%\label{c31}
q_i^{-1}(s) \;=\;\left\{
\begin{array}{ll}
\left(\frac{r_{0,i}+\dr_{0}}{r_{0,i}}\;s\right)^\frac{r_{0,i}}{r_{0,i}+\dr_{0}}\ , &\quad |s|\leq \frac{r_{0,i}+\dr_{0}}{r_{0,i}}\ ,\\
\left(\left(s - \frac{r_{0,i}}{r_{0,i}+\dr_{0}} +  \frac{r_{\infty,i}}{r_{\infty,i}+\dr_{\infty}}\right)\frac{r_{\infty,i}+\dr_{\infty}}{r_{\infty,i}}\right)^\frac{r_{\infty,i}}{r_{\infty,i}+\dr_{\infty}}\ ,&\quad |s|\geq \frac{r_{0,i}+\dr_{0}}{r_{0,i}}\ .
\end{array}
 \right.
$$
%It is $C^1$ on $\RR\setminus\{0\}$.
By (\ref{c28}) the function~:
\begin{equation}\label{b93}
s\;\mapsto\;q_i^{-1}(s)^{\frac{d_{W_0}-r_{0,i}}{r_{0,i}}}\,+\, q_i^{-1}(s)^{\frac{d_{W_\infty}-r_{\infty,i}}{r_{\infty,i}}}
\end{equation}
is homogeneous  in the bi-limit with associated approximating functions $\left(\frac{r_{0,i}+\dr_{0}}{r_{0,i}}s\right)^\frac{d_{W_0}-r_{0,i}}{r_{0,i}+\dr_{0}}$
and $\left(\frac{r_{\infty,i}+\dr_{\infty}}{r_{\infty,i}}s\right)^\frac{d_{W_\infty}-r_{\infty,i}}{r_{\infty,i}+\dr_{\infty}}$.
Furthermore, by (\ref{c29}), it is $C^1$ on $\RR$ and its derivative is homogeneous in the bi-limit with continuous approximating functions \\[0.5em]
$s\mapsto\frac{d_{W_0}-r_{0,i}}{r_{0,i}}\left|\frac{d_{W_0}-r_{0,i}}{r_{0,i}+\dr_0}s\right|^\frac{d_{W_0}-2r_{0,i}-\dr_0}{r_{0,i}+\dr_0}$\hfill,\hfill$
s\mapsto\frac{d_{W_\infty}-r_{\infty,i}}{r_{\infty,i}}\left|\frac{d_{W_\infty}-r_{\infty,i}}{r_{\infty,i}+\dr_\infty}s\right|^\frac{d_{W_\infty}-2r_{\infty,i}-\dr_\infty}{r_{\infty,i}+\dr_\infty}\ .$\\[0.5em]
Let $\barW _i~:\RR^{n-i+1}\rightarrow\RR_+$ be defined by
\\[1em]\null \quad
$\displaystyle
\barW _i(E_{i+1}, s) \;=\; W_{i+1}(E_{i+1}) \;+\;\int^{s}_{q_i^{-1}(e_{i+1})} \left( h^{\frac{d_{W_0}-r_{0,i}}{r_{0,i}}}\;+\;
h^{\frac{d_{W_\infty}-r_{\infty,i}}{r_{\infty,i}}} \right)dh
$\hfill\null\\[0.6em]
\null\hfill$\displaystyle
-\int^{s}_{q_i^{-1}(e_{i+1})}\left(
q_i^{-1}(e_{i+1})^{\frac{d_{W_0}-r_{0,i}}{r_{0,i}}}
+
q_i^{-1}(e_{i+1})^{\frac{d_{W_\infty}-r_{\infty,i}}{r_{\infty,i}}} \right)dh .
\null $\\[1em]
This function is $C^1$ and by (\ref{c28}), Proposition \ref{b94} yields that it is homogeneous in the bi-limit with weights $(r_{0,i+1}, \dots, r_{0,n})$ and $(r_{\infty,i+1}, \dots, r_{\infty,n})$ for $E_{i+1}$,  $r_{0,i}$ and $r_{\infty,i}$ for $s$, and degrees $d_{W_0}$ and $d_{W_\infty}$.
Furthermore, for each $j$ in $\{i+1, \dots, n\}$, the functions $\frac{\partial \barW _i}{\partial e_{j}}(E_{i+1}, s)$ are also homogeneous in the bi-limit with the same weights, and degrees $d_{W_0}-r_{0,j}$ and $d_{W_\infty}-r_{\infty,j}$.

\vspace{1em}\noindent
\textbf{2. Construction of the vector field $K_i$~:}
Given a positive real number $\ell$, we define the vector field $K_i~:\RR^{n-i}\rightarrow\RR^{n-i}$ as~:
\\[0.410em]\null \hfill $\displaystyle
$$K_i(e_i)=\left(\begin{array}{c}-q_i(\ell
e_i)\\K_{i+1}(q_i(\ell e_i))
\end{array}\right)$\hfill \null \\[0.410em]
By Propositions \ref{prop1} and the properties we have established for $q_i$, $K_i$ is a homogeneous in the
bi-limit vector field.
We show now that  selecting $\ell$ large enough yields the asymptotic stability properties.
To begin with, note that for all $E_i = (E_{i+1}, e_i)$ in $\RR^{n-i}$~:
$$
\frac{\partial \barW _i(E_{i+1}, \ell e _{i})}{\partial E_i}(E_i)\left(\SR_{n-i+1}\,  E_{i} + K_{i}(e_{i})\right)\;\leq\; T_1(E_{i+1}, \ell e _{i})\;-\;\ell\,T_2(E_{i+1}, \ell e _{i})\ ,
$$
with the functions $T_1$ and $T_2$ defined as~:
\\[1em]$
T_1(E_{i+1}, \bare _{i}) \,=\,  \displaystyle \frac{\partial \barW _{i}}{\partial
E_{i+1}}(E_{i+1}, \bare _{i})\;\left(\SR_{n-i}\,  E_{i+1} + K_{i+1}(q_i(\bare _{i}))\right)
\  ,
$\hfill \null \\[0.7em]$
T_2(E_{i+1}, \bare _{i}) \,=\,
\left(\bare _{i}^\frac{d_{W_0}-r_{0,i}}{r_{0,i}}
- q_i^{-1}(e_{i+1})^\frac{d_{W_0}-r_{0,i}}{r_{0,i}}
+ \bare _{i}^\frac{d_{W_\infty}-r_{\infty,i}}{r_{\infty,i}}
- q_i^{-1}(e_{i+1})^\frac{d_{W_\infty}-r_{\infty,i}}{r_{\infty,i}}\right)
$\hfill \null \\\null \hfill $
\times(q_i(\bare _{i}) - e_{i+1}) \ .
$
\\[1em]
These functions are homogeneous in the bi-limit with weights
$(r_{\infty,i},\dots,r_{\infty,n})$ and $(r_{0,i},\dots, r_{0,n})$, degrees $\dr_0 + d_{W_0}$ and $\dr_\infty+d_{W_\infty}$, continuous approximating functions
\begin{eqnarray*}
T_{1,0}(E_{i+1}, \bare _{i}) &=& \frac{\partial \barW
_{i,0}}{\partial E_{i+1}}(E_{i+1}, \bare _{i})\;\left(\SR_{n-i}\,  E_{i+1} + K_{i+1,0}(q_{i,0}(\bare _{i}))\right)\ ,\\
T_{1,\infty}(E_{i+1}, \bare _{i}) &=& \frac{\partial \barW
_{i,\infty}}{\partial E_{i+1}}(E_{i+1}, \bare
_{i})\;\left(\SR_{n-i}\,  E_{i+1} + K_{i+1,\infty}(q_{i,\infty}(\bare _{i}))\right)\ ,
\end{eqnarray*}
and
\begin{eqnarray*}
T_{2,0}(E_{i+1}, \bare _{i}) &=&
\left(\bare _{i}^\frac{d_{W_0}-r_{0,i}}{r_{0,i}} - q_{i,0}^{-1}(e_{i+1})^\frac{d_{W_0}-r_{0,i}}{r_{0,i}} \right)(q_{i,0}(\bare _{i}) - e_{i+1})\ ,\\
T_{2,\infty}(E_{i+1}, \bare _{i}) &=&
\left(\bare _{i}^\frac{d_{W_\infty}-r_{\infty,i}}{r_{\infty,i}} - q_{i,\infty}^{-1}(e_{i+1})^\frac{d_{W_\infty}-r_{\infty,i}}{r_{\infty,i}} \right)(q_{i,\infty}(\bare _{i}) - e_{i+1})\ .\end{eqnarray*}
As the function $q_i^{-1}$ is continuous, strictly increasing and onto, the
function
$$
\bare _{i}^\frac{d_{W_0}-r_{0,i}}{r_{0,i}} - q_i^{-1}(e_{i+1})^\frac{d_{W_0}-r_{0,i}}{r_{0,i}}
+
\bare _{i}^\frac{d_{W_\infty}-r_{\infty,i}}{r_{\infty,i}} - q_i^{-1}(e_{i+1})^\frac{d_{W_\infty}-r_{\infty,i}}{r_{\infty,i}}
$$
has a unique zero at $q_i(\bare _{i})=e_{i+1}$ and has the same sign as
$q_i(\bare _{i})-e_{i+1}$. It follows that~:
\begin{eqnarray*}
T_2(E_{i+1}, \bare _{i}) &\geq & 0
\qquad \forall (E_{i+1}, \bare _{i}) \in  \RR^{n-i}\ ,
\\
T_2(E_{i+1},\bare _{i})&=&0\qquad\Rightarrow\qquad q_i(\bare _{i})\;=\;e_{i+1}\ .
\end{eqnarray*}
On the other hand, for all $E_{i}\,\neq\,0$,
%startmodif
$$
T_{1}(E_{i+1}, q_i^{-1}(e_{i+1}))
\;=\; \frac{\partial W_{i+1}}{\partial
E_{i+1}}(E_{i+1})\;\left(\SR_{n-i}\,  E_{i+1}
+ K_{i+1}(e_{i+1})\right) < 0
\ .
$$
Hence (\ref{b41}) yields~:
\\[1em]$
\left\{(E_{i+1},\bare _{i})\in\RR^{n-i+1}\setminus\{0\}\: :\;
T_2(E_{i+1},\bare _{i})\;=\;0 \right\}
$\hfill \null \\[0.7em]\null\hfill$
\subseteq\qquad
\left\{(E_{i+1},\bare _{i})\in\RR^{n-i+1}\: :\; T_{1}(E_{i+1},\bare
_{i})\;<\;0\right\}
\  .
$\\[1em]
By following
%stopmodif
the same argument,
%startmodif
it can be shown
%stopmodif
that this property holds also for the homogeneous approximations, i.e.~:
%startmodif
\\[1em]$
\left\{(E_{i+1},\bare _{i})\in\RR^{n-i+1}\setminus\{0\}\: :\;
T_{2,0}(E_{i+1},\bare _{i})\;=\;0 \right\}
$\hfill \null \\[0.7em]\null\hfill$
\subseteq\qquad
\left\{(E_{i+1},\bare _{i})\in\RR^{n-i+1}\: :\; T_{1,0}(E_{i+1},\bare
_{i})\;<\;0\right\}
\  ,
$\\[0.7em]$
\left\{(E_{i+1},\bare _{i})\in\RR^{n-i+1}\setminus\{0\}\: :\;
T_{2,\infty }(E_{i+1},\bare _{i})\;=\;0 \right\}
$\hfill \null \\[0.7em]\null\hfill$
\subseteq\qquad
\left\{(E_{i+1},\bare _{i})\in\RR^{n-i+1}\: :\; T_{1,\infty }(E_{i+1},\bare
_{i})\;<\;0\right\}
\  .
$\\[1em]
%stopmodif
%
Therefore, by Lemma \ref{3},  there exists $\ell^*$ such
that, for all $\ell \geq  \ell^*$ and all
$(E_{i+1},\bare _{i})\,\neq\,0$~:
\begin{eqnarray*}
T_1(E_{i+1}, \bare _{i})\;-\;\ell\,T_2(E_{i+1}, \bare _{i}) &<& 0\ ,\\
T_{1,0}(E_{i+1}, \bare _{i})\;-\;\ell\,T_{2,0}(E_{i+1}, \bare _{i}) &<& 0\ ,\\
T_{1,\infty}(E_{i+1}, \bare _{i})\;-\;\ell\,T_{2,\infty}(E_{i+1}, \bare _{i}) &<& 0\ .
\end{eqnarray*}
This implies that the origin
is a globally  asymptotically stable equilibrium of the systems (\ref{LP1}),

which concludes the proof.
\end{proof}

To construct the function $K_1$, which defines the observer (\ref{23}), it is sufficient to iterate the
construction proposed in Theorem \ref{theo2} starting from
$$
K_n(e_n) \;=\;-\left\{
\begin{array}{ll}
\frac{1}{1+\dr_{0}}\;(\ell_ne_n)^{1+\dr_{0}} \ ,&\quad |\ell_ne_n|\leq 1\ ,\\[0.1em]
\frac{1}{1+\dr_{\infty}}(\ell_ne_n)^{1+\dr_{\infty}}+ \frac{1}{1+\dr_{0}} -\frac{1}{1+\dr_{\infty}}  \ ,&\quad |\ell_ne_n|\geq 1\ .
\end{array}
 \right.
$$
where $\ell_n$ is any strictly positive positive real number.
Indeed, $K_n$ is a homogeneous in the bi-limit vector field with approximating functions
$K_{n,0}(e_n)=\frac{1}{1+\dr_{0}}\;(\ell_ne_n)^{1+\dr_{0}}$ and
$K_{n, \infty}(e_n)=\frac{1}{1+\dr_{\infty}}(\ell_ne_n)^{1+\dr_{\infty}}$.
This selection implies that

the origin is a globally asymptotically stable of the systems $\dot e_n \;=\; K_n(e_n)$,
$\dot e_n \;=\; K_{n, 0}(e_n)$ and
$\dot e_n \;=\; K_{n, \infty}(e_n)$.

Consequently the assumptions of Theorem \ref{theo2} are
satisfied for $i+1=n$. We can apply it recursively up to $i=1$
obtaining the vector field $K_1$.

As a result of this procedure we obtain a homogeneous in the bi-limit observer which globally asymptotically observes the state of the system (\ref{5}), and also the state for its homogeneous approximations around the origin and at infinity.
In other word,

the origin is a globally asymptotically stable equilibrium of the systems

\begin{equation}\label{b37}
\dot E_1 = \SR_n  E_1 + K_1(e_1)\; ,\;
\dot E_1 = \SR_n  E_1 + K_{1,0}(e_1) \; ,\;
\dot E_1 = \SR_n  E_1 + K_{1,\infty}(e_1)\ .
\end{equation}

\remark \label{c32}
Note that when
%startmodif
$0\leq \dr_0  \leq \dr_\infty$,
%stopmodif
we have
%startmodif
$1\leq \frac{r_{0,i}+\dr_{0}}{r_{0,i}} \leq \frac{r_{\infty,i}+\dr_{\infty}}{r_{\infty,i}}$
%stopmodif
 for $i=1\dots,n$ and we can replace the function $q_i$ in (\ref{c31}) by the simpler function~:
$$
q_i(s)\;=\; s^\frac{r_{0,i}+\dr_{0}}{r_{0,i}}\;+\; s^\frac{r_{\infty,i}+\dr_{\infty}}{r_{\infty,i}}
$$
%startmodif
which has been used already in \cite{Andrieu-Praly-Astolfi-06}.
%stopmodif

\example
Consider a chain of integrators of dimension two, with the following
weights and degrees~:
% (see (\ref{b56}) and (\ref{b45}))~:
$$
\left(r_0, \, \dr_0\right) \;=\;  \Big((2-q,1),\,q-1\Big)
\quad ,\qquad
\left( r_\infty,\, \dr_\infty\right) \;=\; \Big((2-p,1),\, p-1 \Big)\ .
$$
When $q\geq p$ (i.e. $\dr_0\leq  \dr_\infty$), by following the above recursive observer design
we obtain two positive
%startmodif
real numbers
%stopmodif
 $\ell_1$ and $\ell_2$ such that the system~:
$$
\dot {\hat \chi}_1 \;=\; \hat \chi_{2} - q_1(\ell_1 e_1)
\  ,\quad
\dot {\hat \chi}_{2} \;=\; u - q_2(\ell_2\,q_1(\ell_1e_1 ))\  ,
\  ,\quad
e_1 \;=\; \hat \chi_1 - y
\  .
$$
with,
\\[1em]
\refstepcounter{equation}\label{c48}  (\theequation)
\hfill $
q_2(s) =\left\{
\begin{array}{@{}ll@{}}
\frac{1}{q}\;s^q &,\, |s|\leq 1\\[0.2em]
\frac{1}{p}s^p + \frac{1}{q}-\frac{1}{p}&,\, |s|\geq 1
\end{array}
 \right.
\  ,\quad
q_1(s) =\left\{
\begin{array}{@{}ll@{}}
(2-q)\;s^\frac{1}{2-q} &,\, |s|\leq 1\\[0.2em]
(2-p)s^\frac{1}{2-p} +p-q&,\, |s|\geq 1
\end{array}\right.\ .
$\hfill \null \\[1em]
is a global observer for the system $\dot \chi_1=\chi _2\; ,\  \dot
\chi _2= u\; ,\  y=\chi_1$.
% (\ref{b18}).
Furthermore, its homogeneous approximations around the origin and at infinity
%namely
%$$
%\left \{\begin{array}{rcl}
%\dot {\hat \chi}_1 &=&\hat \chi_{2} - (2-q)[\ell_1 e_1]^\frac{1}{2-q}\\
%\dot {\hat \chi}_{2} &=&u -
%\left[\vphantom{a\over a}\ell_2 [\ell_1 e_1]^\frac{1}{2-q} \right]^q
%\end{array}\right.\quad,\qquad \left \{\begin{array}{rcl}
%\dot {\hat \chi}_1 &=&\hat \chi_{2} - [\ell_1 e_1]^\frac{1}{2-p}\\
%\dot {\hat \chi}_{2} &=&u -
%\left[\vphantom{a\over a}\ell_2 [\ell_1 e_1]^\frac{1}{2-p}\right]^p
%\end{array}\right.\ ,
%$$
are also global observers for the same system.

\section{Recursive design of a homogeneous in the bi-limit state feedback}
\label{b99}
It is well-known that the system (\ref{5}) can be rendered
homogeneous by using a stabilizing homogeneous
state feedback which can be designed by backstepping (see
\cite{Praly-AndreaNovel-Coron,Praly-Mazenc,Morin-Samson,Qian,Tzamtzi-Tsinias,Hong}
for
instance).
We show in this section that this property can be extended to the
case of homogeneity in the bi-limit. More precisely, we show that there exists a homogeneous in the bi-limit function
$\phi_n$ such that the system (\ref{5}) with $u\,=\,\phi_n(\XR_n)$ is homogeneous in the bi-limit,
with weights $r_0$ and $r_\infty$ and degrees $\dr_0$ and $\dr_\infty$.
Furthermore, its origin and the origin of the approximating systems in the $0$-limit and in the
$\infty$-limit are globally asymptotically stable equilibria.

To design the state feedback we follow the approach of Praly and Mazenc \cite{Praly-Mazenc}.
To this end, consider the auxiliary system with state
$\XR_i\,=\,(\chi_1,\dots,\chi_i)$ in $\RR^i$, $1\leq i<n$, and dynamics~:
\begin{equation}\label{b33}
\null \qquad
\dot \chi_1 \,=\, \chi_2\ ,\
\dots\ , \
\dot \chi_i \,=\, u\qquad \textrm{ or in compact form }\quad \dot \XR_i \;=\;
\SR_{i}\,  \XR_i\,+\, B_i\,  u\ .
\end{equation}
where $u$ is the input in $\RR$, $\SR_i$ is the shift matrix of order $i$ i.e.
$\SR_{i}\,\XR_{i} = \left(\chi_2,\dots,\chi_i,0\right)^T$, and
$B_i=(0,\dots,1)^T$ is in $\RR^i$.
We show that, if there exists a homogeneous in the bi-limit
stabilizing control law for the origin of the system (\ref{b33}), then there is one for the origin the system with state
$\XR_{i+1}\,=\,(\chi_1,\dots,\chi_{i+1})$ in $\RR^{i+1}$ defined by~:
\begin{equation}
\dot \chi_1 \,=\, \chi_2\ ,\
\dots\ , \
\dot \chi_{i+1} \,=\, u\qquad , \textrm{ i.e. }\qquad \dot \XR_{i+1} \;=\;
\SR_{i+1}\,  \XR_{i+1}\,+\, B_{i+1}\,  u\ .
\end{equation}
Let $\dr_0$ and $\dr_\infty $ be in $(-1, \frac{1}{n-1})$ and consider the weights and degrees defined in (\ref{b60}).

\begin{theorem}[Homogeneous in the bi-limit backstepping]
\label{theo1}
Suppose there exists a homogeneous in the bi-limit function $\phi_i~:\RR^i\rightarrow\RR$ with associated triples $(r_0, \dr_0+r_{0,i}, \phi_{i,0})$ and $(r_\infty,\dr_\infty+r_{\infty,i},\phi_{i,\infty})$ such that the following holds.
\begin{enumerate}
\item There exist $\alpha_{i} \geq  1$ such that the function
$\psi_{i}(\XR_i)\,=\,\phi_i(\XR_i)^{\alpha_i}$ is $C^1$ and for each $j$ in $\{1, \dots, i\}$
the function
$\frac{\partial \psi_i}{\partial \chi_j}$
is homogeneous in the bi-limit,
with weights $(r_{0,1},\dots,r_{0,i})$, $(r_{\infty,1},\dots,r_{\infty,i})$, degrees $\alpha_{i}(r_{0,i}+\dr_0)-r_{0,j}$ and $\alpha_i(r_{\infty,i}+\dr_\infty)-r_{\infty,j}$ and approximating functions $\frac{\partial \psi_{i0}}{\partial \chi_j}$, $\frac{\partial \psi_{i\infty}}{\partial \chi_j}$.

\item The origin   is a globally asymptotically stable equilibrium of the systems  ~:
\\[0.7em]\null \hskip -\leftmargin\refstepcounter{equation}\label{c22}(\theequation)
\hfill $
\dot \XR_{i} =
\SR_{i}\,  \XR_i + B_i\,  \phi_i(\XR_{i})
\  ,\quad
\dot \XR_{i} =
\SR_{i}\,  \XR_i + B_i\,  \phi_{i,0}(\XR_{i})
\  ,\quad
\dot \XR_{i} =
\SR_{i}\,  \XR_i + B_i\,  \phi_{i,\infty}(\XR_{i})\ .
$\hfill \null \\[0.7em]
% \begin{eqnarray}
% \nonumber
% \dot \XR_{i} &=&
% \SR_{i}\,  \XR_i\,+\, B_i\,  \phi_i(\XR_{i})
% \  ,\\\label{c22}
% \dot \XR_{i} &=&
% \SR_{i}\,  \XR_i\,+\, B_i\,  \phi_{i,0}(\XR_{i})
% \  ,\\
% \nonumber
% \dot \XR_{i} &=&
% \SR_{i}\,  \XR_i\,+\, B_i\,  \phi_{i,\infty}(\XR_{i})
% \end{eqnarray}
\end{enumerate}
Then there exits a homogeneous in the bi-limit function
$\phi_{i+1}~:\RR^{i+1}\rightarrow\RR$ with associated triples $(r_0,\dr_0+r_{0,i+1},\phi_{i+1,0})$ and $(r_\infty,\dr_\infty+r_{\infty,i+1},\phi_{i+1,\infty})$ such that the same properties
hold, i.e.
\begin{enumerate}
\item There exists a real number $\alpha_{i+1}>1$ such that the
function $\psi_{i+1}(\XR_{i+1})\,=\,\phi_{i+1}(\XR_{i+1})^{\alpha_{i+1}}$ is $C^1$ and for each $j$ in $\{1, \dots, i+1\}$ the function
$\frac{\partial \psi_{i+1}}{\partial \chi_j}$ is homogeneous in the bi-limit
with weights $(r_{0,1},\dots,r_{0,i+1})$, $(r_{\infty,1},\dots,r_{\infty,i+1})$, degrees $\alpha_{i+1}(r_{0,i+1}+\dr_0)-r_{0,j}$ and $\alpha_{i+1}(r_{\infty,i+1}+\dr_\infty)-r_{\infty,j}$ and approximating functions $\frac{\partial \psi_{i+1,0}}{\partial \chi_j}$, $\frac{\partial \psi_{i+1,\infty}}{\partial \chi_j}$.

\item The origin

is a  globally asymptotically stable equilibrium of the systems

% \\[0.7em]\null \hskip -\leftmargin\refstepcounter{equation}\label{b42}(\theequation)
% \hfill $
% \XR_{i+1} =
% \SR_{i+1}\,  \XR_{i+1} + B_{i+1}\,  \phi_{i+1}(\XR_{i+1})
% \  ,\quad
% \XR_{i+1} =
% \SR_{i+1}\,  \XR_{i+1} + B_{i+1}\, \phi_{i+1,0}(\XR_{i+1})
% \  ,\quad
% \XR_{i+1} =
% \SR_{i+1}\,  \XR_{i+1} + B_{i+1}\, \phi_{i+1,\infty}(\XR_{i+1})
% $\hfill \null \\[0.7em]
\begin{eqnarray}
\nonumber
\XR_{i+1} &=&
\SR_{i+1}\,  \XR_{i+1}\,+\, B_{i+1}\,  \phi_{i+1}(\XR_{i+1})
\  ,
\\\label{b42}
\XR_{i+1} &=&
\SR_{i+1}\,  \XR_{i+1}\,+\, B_{i+1}\, \phi_{i+1,0}(\XR_{i+1})
\  ,
\\\nonumber
\XR_{i+1} &=&
\SR_{i+1}\,  \XR_{i+1}\,+\, B_{i+1}\, \phi_{i+1,\infty}(\XR_{i+1})\ .
\end{eqnarray}

\end{enumerate}
\end{theorem}

\begin{proof}
We prove this result in three steps.
First we construct a homogeneous in the bi-limit Lyapunov function, then we define a control law parametrized by a real number $k$.
Finally we show that there exists $k$ such that the time derivative,  along the trajectories of the systems (\ref{b42}), of the Lyapunov function and of its approximating functions are negative definite.

\vspace{1em}\noindent
\textbf{1. Construction of the Lyapunov function.}
Let $d_{V_0}$ and $d_{V_\infty}$ be positive real numbers satisfying~:
%\begin{equation}
%\null \qquad d_{V_\infty} \geq d_{V_0} > (1+\alpha_i)\,r_{0,i+1}\quad\textrm{and}\quad
%d_{V_\infty} \geq d_{V_0} >r_{0,j}\qquad \forall \; j\;\in\;
%\{1,\ldots,n\}\ ,
%\end{equation}
%
\begin{equation}\label{b65}
d_{V_0} > \max_{j\in\{1,\ldots,n\}}\{r_{0,j}\}\quad, \qquad d_{V_\infty} > \max_{j\in\{1,\ldots,n\}}\{r_{\infty,j}\}\ ,
\end{equation}
and
\begin{equation}\label{c33}
\frac{d_{V_\infty}}{r_{\infty,i+1}} \;\geq \;\frac{d_{V_0}}{r_{0,i+1}} > 1+\alpha_i\ .
\end{equation}
%As $d_0>-1$ and $d_\infty>-1$, the selection (\ref{b60}) yields $r_{0,j}+d_{0}>0$ and $r_{\infty,j}+d_{\infty}>0$ for each $1\,\leq\,j\leq\,n$.
%Hence,
With this selection,
%startmodif
Theorem \ref{b26} gives the existence of a
%stopmodif
 $C^1$, proper and positive definite function
$V_i:\RR^i\rightarrow\RR_+$ such that, for each $j$ in  $\{1, \dots, n\}$, the function
$\frac{\partial V_i}{\partial \chi_j}$ is homogeneous in the bi-limit with
weights $(r_{0,1},\dots,r_{0,i})$,
$(r_{\infty,1},\dots,r_{\infty,i})$, degrees $d_{V_0}-r_{0,j}$, $d_{V_\infty}-r_{\infty,j}$, and approximating functions $\frac{\partial V_{i,0}}{\partial \chi_j}$ and
$\frac{\partial V_{i,\infty}}{\partial \chi_j}$.
Moreover, we have for all $\XR_i \in\RR^{i}\setminus\{0\}$~:
\begin{eqnarray}
\nonumber
\frac{\partial V_{i}}{\partial \XR_{i}}(\XR_{i})\,
\left[\SR_{i}\,  \XR_i +  B_i\,\phi_i(\XR_{i}) \right]&<& 0
\  ,
\\
\label{b35}
\frac{\partial V_{i,0}}{\partial \XR_{i}}(\XR_{i})\,
\left[
\SR_{i}\,  \XR_i + B_i\,\phi_{i,0}(\XR_{i})
\right]&<& 0
\  ,
\\\nonumber
\displaystyle\frac{\partial V_{i,\infty}}{\partial \XR_{i}}(\XR_{i})\,
\left[\SR_{i}\,  \XR_i + B_i\,\phi_{i,\infty}(\XR_{i})\right]&<& 0
\  .
\end{eqnarray}
%startmodif
Following \cite{Praly-AndreaNovel-Coron},
%stopmodif
consider the Lyapunov function $V_{i+1}~:\RR^{i+1}\rightarrow\RR_+$ defined by~:
\\[1em]$\displaystyle
V_{i+1}(\XR_{i+1})\;=\; V_i(\XR_i) \;+\; \int_{\phi_i(\XR_i)}^{\chi_{i+1}} \left(h^\frac{d_{V_0}-r_{0,i+1}}{r_{0,i+1}}\,-\, \phi_i(\XR_{i})^\frac{d_{V_0}-r_{0,i+1}}{r_{0,i+1}}\right)\;dh
$\hfill\null\\
\null\hfill$\displaystyle
\,+\,\int_{\phi_i(\XR_i)}^{\chi_{i+1}} \left(h^\frac{d_{V_\infty}-r_{\infty,i+1}}{r_{\infty,i+1}}\,-\, \phi_i(\XR_{i})^\frac{d_{V_\infty}-r_{\infty,i+1}}{r_{\infty,i+1}}\right)
 \;dh
\  .
$\\[1em]
This function is positive definite and proper.
Furthermore, as $d_{V_\infty}$ and $d_{V_0}$ satisfy (\ref{c33}) we
have~:
$$\frac{d_{V_\infty}-r_{\infty,i+1}}{r_{\infty,i+1}}\geq\frac{d_{V_0}-r_{0,i+1}}{r_{0,i+1}} > \alpha_i \geq 1
\  .
$$
The function $\psi_i(\XR_i)=\phi_i(\XR_i)^{\alpha_i}$ being $C^1$, this inequality yields that the
function $V_{i+1}$ is $C^1$ .
Finally, for each $j$ in $\{1,\dots,n\}$, the function
$\frac{\partial V_{i+1}}{\partial \chi_j}$ is homogeneous in the bi-limit
with associated triples
\\
$\left((r_{0,1},\dots,r_{0,i+1}),
d_{V_0}-r_{0,j},
\frac{\partial V_{i+1,0}}{\partial \chi_j} \right)$,
$\left((r_{\infty,1},\dots,r_{\infty,i+1}),
d_{V_\infty}-r_{\infty,j},
\frac{\partial V_{i+1,\infty}}{\partial \chi_j} \right)$

\par\vspace{1em}\noindent
\textbf{2. Definition of the control law~:}
%startmodif
Recall (\ref{c41}) and
%stopmodif
consider the function $\psi_{i+1}~:\RR^{i+1}\rightarrow\RR$ defined by~:
$$
\psi_{i+1}(\XR_{i+1})=
-k\int_0^{\chi_{i+1}^{\alpha_i} - \phi_i(\XR_i)^{\alpha_i}}
\mathfrak{H}
\left(
|s|^{\alpha_{i+1}\frac{\dr_0+r_{0,i+1}}{\alpha_{i}\,r_{0,i+1}}-1}
,
|s|^{\alpha_{i+1}\frac{\dr_\infty+r_{\infty,i+1}}{\alpha_{i}\,r_{\infty,i+1}}-1}
\right)\;ds\ ,
$$
where $k$, in $\RR_+$, is a design parameter and $\alpha_{i+1}$ is selected as
$$
\alpha_{i+1}\;\geq\;\max\left\{ \frac{\alpha_{i}\,r_{0,i+1}}{\dr_0+r_{0,i+1}}, \frac{\alpha_{i}\,r_{\infty,i+1}}{\dr_\infty+r_{\infty,i+1}},1\right\}\ .
$$
%startmodif
$\psi_{i+1}$ takes values with  the same sign as
$\chi_{i+1}-\phi_i(\XR_{i})$, it
is
%stopmodif
$C^1$ and, by Proposition \ref{b94}, it is homogeneous in the bi-limit.
Furthermore, by Proposition \ref{prop1}, for each $j$ in $\{1, \dots, i+1\}$, the function
$\frac{\partial \psi_{i+1}}{\partial \chi_j}$ is homogeneous in the bi-limit,
with weights $(r_{0,1},\dots,r_{0,i+1})$, $(r_{\infty,1},\dots,r_{\infty,i+1})$, degrees $\alpha_{i+1}(r_{0,i+1}+\dr_0)-r_{0,j}$ and $\alpha_{i+1}(r_{\infty,i+1}+\dr_\infty)-r_{\infty,j}$ and approximating functions $\frac{\partial \psi_{i+1,0}}{\partial \chi_j}$, $\frac{\partial \psi_{i+1,\infty}}{\partial \chi_j}$.
%startmodif
With this at hand, we choose the control law $\phi_{i+1}$ as~:
%stopmodif
%
$$
\phi_{i+1}(\XR_{i+1})\;=\;\psi_{i+1}(\XR_{i+1})^\frac{1}{\alpha_{i+1}}
$$

\par\vspace{1em}\noindent
\textbf{3. Selection of $k$.}
Note that~:
\begin{equation}
\frac{\partial V_{i+1}}{\partial \XR_{i+1}}(\XR_{i+1})\,
\left[\SR_{i+1}\,  \XR_{i+1} + B_{i+1}\,  \phi_{i+1}(\XR_{i+1})\right]
 \;=\; T_1(\XR_{i+1}) - k\,T_2(\XR_{i+1})\ ,
\end{equation}
with the functions $T_1$ and $T_2$ defined as~:
\\[0.5em]
$\displaystyle
%startmodif
T_1(\XR_{i+1})\,=\,\frac{\partial V_{i+1}}{\partial \XR_i}(\XR_{i+1})\,
\left[\SR_{i}\,  \XR_{i} + B_{i} \chi_{i+1})\right]
%stopmodif
$\hfill\null
\\[0.5em]
$\displaystyle
T_2(\XR_{i+1})\,=\,\left(\chi_{i+1}^\frac{d_{V_0}-r_{0,i+1}}{r_{0,i+1}}\,-\, \phi_i(\XR_{i})^\frac{d_{V_0}-r_{0,i+1}}{r_{0,i+1}}\right.$\hfill\null\\%[0.6em]
\null\hfill$\displaystyle\left.\;+\;
\chi_{i+1}^\frac{d_{V_\infty}-r_{\infty,i+1}}{r_{\infty,i+1}}\,-\, \phi_i(\XR_{i})^\frac{d_{V_\infty}-r_{\infty,i+1}}{r_{\infty,i+1}}\right)
\phi_{i+1}(\XR_{i+1})\ .
$
\\[0.5em]
By definition of homogeneity in the bi-limit and Proposition \ref{prop1},
these functions are homogeneous in the bi-limit with weights $(r_{0,1},\dots,r_{0,i+1})$ and $(r_{\infty,1},\dots,r_{\infty,i+1})$, and degrees $d_{V_0}+\dr_0$ and $d_{V_\infty}+\dr_\infty$.
Moreover, since $\phi_{i+1}(\XR_{i+1})$ has the same sign as $\chi_{i+1}-\phi_i(\XR_{i})$,
$T_2(\XR_{i+1}) $ is non-negative for all $\XR_{i+1}$ in $\RR^{i+1}$
and as $\phi_{i+1}(\XR_{i+1})=0$ only if $\chi_{i+1}-\phi_i(\XR_{i})=0$ we get~:
\begin{eqnarray*}
T_2(\XR_{i+1}) = 0 & \quad  \Longrightarrow \quad   &\chi_{i+1}\,=\, \phi_i(\XR_i)\  ,
\\
\chi_{i+1}\,=\, \phi_i(\XR_i)  & \quad  \Longrightarrow \quad  &
T_1(\XR_{i+1})\,=\,\frac{\partial V_i}{\partial \XR_i}(\XR_i)
\left[\SR_{i}\,  \XR_{i} + B_i\phi_i(\XR_i) \right]
\  .
\end{eqnarray*}
Consequently, equations (\ref{b35}) yield~:
%startmodif
$$
\left\{\XR_{i+1}\in \RR^{i+1}\setminus\{0\}\: :\;  T_2(\XR_{i+1}) =0 \right\}
\quad \subseteq\quad
\left\{\XR_{i+1}\in \RR^{i+1}\: :\;  T_1(\XR_{i+1}) < 0 \right\}\  .
$$
%stopmodif
The same implication holds for the homogeneous approximations of the two functions at infinity and around the origin, i.e.
%startmodif
\begin{eqnarray*}
\left\{\XR_{i+1}\in \RR^{i+1}\setminus\{0\}\: :\;  T_{2,0}(\XR_{i+1}) =0 \right\}
\quad \subseteq\quad
\left\{\XR_{i+1}\in \RR^{i+1}\: :\;  T_{1,0}(\XR_{i+1}) < 0 \right\}
\ ,\\
\left\{\XR_{i+1}\in \RR^{i+1}\setminus\{0\}\: :\;  T_{2,\infty }(\XR_{i+1}) =0 \right\}
\quad \subseteq\quad
\left\{\XR_{i+1}\in \RR^{i+1}\: :\;  T_{1,\infty }(\XR_{i+1}) < 0 \right\}
\ .
\end{eqnarray*}
%stopmodif
Hence, by Lemma \ref{3}, there exists
$k^*>0$ such that, for all $k\geq k^*$, we have for all $\XR_{i+1}\,\neq\,0$~:
\begin{eqnarray*}
\frac{\partial V_{i+1}}{\partial \XR_{i+1}}(\XR_{i+1})\,
\left[\SR_{i+1}\,  \XR_{i+1} + B_{i+1} \phi_{i+1}(\XR_{i+1})\right] &<& 0\ ,\\
\frac{\partial V_{i+1,0}}{\partial \XR_{i+1}}(\XR_{i+1})\,
\left[\SR_{i+1}\,  \XR_{i+1} + B_{i+1} \phi_{i+1,0}(\XR_{i+1})\right] &<& 0
\  ,
\\
\frac{\partial V_{i+1,\infty}}{\partial \XR_{i+1}}(\XR_{i+1})
\left[\SR_{i+1}\,  \XR_{i+1} + B_{i+1}
\phi_{i+1,\infty}(\XR_{i+1})\right] &<& 0\ .
\end{eqnarray*}
This implies that   the origin is a globally
asymptotically stable equilibrium of the systems (\ref{b42}).
\end{proof}

To construct the function $\phi_n$ it is sufficient to iterate the
construction in Theorem \ref{theo1} starting from
$$
\phi_1(\chi_1)\;=\; \psi_1(\chi_1)^\frac{1}{\alpha_1}\quad,\qquad\psi_{1}(\chi_1)=
-k_1\int_0^{\chi_{1}}
\mathfrak{H}
\left(
|s|^{\alpha_{1}\frac{r_{0,2}}{r_{0,1}}-1}
,
|s|^{\alpha_{1}\frac{r_{\infty,2}}{r_{\infty,1}}-1}
\right)\;ds\ ,
$$
with $k_1>0$.

At the end of the recursive procedure, we have that the
origin is a globally asymptotically stable equilibrium of the systems~:
% \begin{equation}
% \label{b38}
% \XR_n  =  \SR_n\,  \XR_n + B_n\,  \phi_n(\XR_n)
% \  ,\
% \XR_n  = \SR_n\,  \XR_n + B_n\, \phi_{n,0}(\XR_{n})
% \  ,\
%  \XR_{n}  = \SR_n\,  \XR_n +  B_n\, \phi_{n,\infty}(\XR_{n})
% \end{equation}
\begin{eqnarray}
\nonumber
\XR_n &=& \SR_n\,  \XR_n\,+\, B_n\,  \phi_n(\XR_n)
\  ,
\\\label{b38}
\XR_n &=& \SR_n\,  \XR_n\,+\, B_n\, \phi_{n,0}(\XR_{n})
\  ,
\\\nonumber
 \XR_{n} &=& \SR_n\,  \XR_n\,+\, B_n\, \phi_{n,\infty}(\XR_{n})\ .
\end{eqnarray}

\remark \label{c35}
 Note that if $\dr_0\geq 0$ and $\dr_\infty \geq 0$, then we can select $\alpha_i=1$ for all $1\leq i \leq n$ and if $\dr_0\leq 0$ and $\dr_\infty \geq \dr_0$ we can select $\alpha_i=\frac{r_{0,1}}{r_{0,i+1}}$. Finally if $\dr_\infty \leq 0$ and $\dr_0 \geq \dr_\infty$ we can select $\alpha_i=\frac{r_{\infty,1}}{r_{\infty,i+1}}$.

\remark \label{c34}
As in the observer design, when $\dr_0  \leq \dr_\infty$, we have $\frac{r_{0,i+1}+\dr_{0}}{r_{0,i+1}} \leq \frac{r_{\infty,i+1}+\dr_{\infty}}{r_{\infty,i+1}}$ for $i=1\dots,n$ and we can replace the function $\psi_i$ by the simpler function~:\\[0.5em]
\refstepcounter{equation}$(\theequation)$\label{c36}
\quad$\psi_{i+1}(\XR_{i+1})\;=\;-k\left(
|\chi_{i+1}^{\alpha_i} - \phi_i(\XR_i)^{\alpha_i}|^{\alpha_{i+1}\frac{\dr_0+r_{0,i+1}}{\alpha_{i}\,r_{0,i+1}}}
\right.$
\\[-0.5em]
\null\hfill$\left.+
|\chi_{i+1}^{\alpha_i} - \phi_i(\XR_i)^{\alpha_i}|^{\alpha_{i+1}\frac{\dr_\infty+r_{\infty,i+1}}{\alpha_{i}\,r_{\infty,i+1}}}
\right)\ .$
\\[0.5em]
Finally if $0\leq \dr_0  \leq \dr_\infty$, taking $\alpha_i=1$ (see Remark \ref{c35}) and $\phi(\XR_{i+1})\,=\,\psi_{i+1}(\XR_{i+1})$ as defined in (\ref{c36}), we recover the design in \cite{Andrieu-Praly-Astolfi-06}.

\example
Consider a chain of integrators of dimension two with weights and degrees~:
%startmodif
$$
\left(r_0, \, \dr_0\right) \;=\;  \Big((2-q,1),\,q-1\Big)
\quad ,\qquad
\left( r_\infty,\, \dr_\infty\right) \;=\; \Big((2-p,1),\, p-1 \Big)\ ,
$$ with $2>p> q> 0$.
%stopmodif
Given $k_1>0$, using the proposed backstepping procedure we obtain a positive real number $k_2$ such that the feedback~:
\begin{equation}\label{c49}
\phi_{2}(\chi_1,\chi_2)=
-k_2\int_0^{\chi_{1} - \phi_i(\chi_1)}
\mathfrak{H}
\left(
|s|^{q-1}
,
|s|^{p-1}
\right)\;ds\ ,
\end{equation}
with $\phi_{1}(\chi_1)=
-k_1\int_0^{\chi_{1}}
\mathfrak{H}
\left(
|s|^{\frac{q-1}{2-q}}
,
|s|^{\frac{p-1}{2-p}}
\right)\;ds$
renders the origin   a globally asymptotically stable equilibrium of the closed loop system.
Furthermore, as a consequence of the robustness result in Corollary \ref{b36}, there is a positive real number
$c_G$ such that, if the positive real numbers $|c_0|$ and $|c_\infty|$ associated with
$\delta_i$ in (\ref{b46}) are smaller than $c_G$, then the control law $\phi_2$ globally
asymptotically stabilizes the origin of system (\ref{b18}).

\section{Application to nonlinear output feedback design}
\label{b70}

\subsection{Results on output feedback}

The tools presented in the previous sections
can be used to derive two new results on stabilization by output feedback for the origin of nonlinear systems.
The output feedback is designed for
a simple chain of integrators~:
\begin{equation}\label{c55}
\dot x \;=\; \SR_n\, x\;+\; B_n\,u \quad, \qquad y = x_1\ ,
\end{equation}
where $x$ is in $\RR^n$,  $y$ is the output in $\RR$, $u$ is the control input in $\RR$.
It is then shown to be adequate to solve the output feedback
stabilization problem for the origin of systems for which this chain of integrators
can be considered as the dominant part of the dynamics.
Such a domination approach has a long history.

It is the corner stone of the results in \cite{Khalil-Saberi} (see also \cite{Qian-Lin}) where a linear controller was introduced to deal with a nonlinear systems.
This approach has also been followed with nonlinear controller   in \cite{Praly-Jiang}
and more recently in combination with weighted homogeneity in \cite{Yang-Lin, Qian, Qian-Lin-2006} and references therein.

In the context of homogeneity in the bi-limit, we use it
exploiting the proposed backstepping and recursive
observer designs.
Following the idea introduced by Qian in \cite{Qian} (see also \cite{Qian-Lin}), the output feedback we proposed
is given by~:
\begin{eqnarray}
\label{b84}
\dot {\hat \XR}_n &=& L\,\left(
\SR_{n}\, \hat  \XR_{n} + B_{n} \phi_n(\hat \XR_n) \,+\, K_1(x_1-\hat \chi_1)\right)
\  ,\quad
u\;=\;L^n\,\phi_n(\hat \XR_n)\ ,
\end{eqnarray}
with $\hat \XR_n$ in $\RR^n$ and where $\phi_n$ and $K_1$ are continuous functions  and $L$ is a positive real number.
Employing the recursive procedure given in Sections \ref{b69} and \ref{b99},
we get the following theorem
whose proof is in section \ref{LP17}.
\begin{theorem}\label{c54}
For every real numbers $\dr_0$ and $\dr_\infty$,
in $\left(-1,\frac{1}{n-1}\right)$,
 there exist a homogeneous in the bi-limit function $\phi_n~:\RR^n\rightarrow\RR$ with associated triples $(r_0, 1 + \dr_0, \phi_{n,0})$ and
 $(r_\infty, 1 + \dr_\infty, \phi_{n,\infty})$
 and a homogeneous in the bi-limit vector field
$K_1~:\RR^n\rightarrow\RR^n$ with associated triples $(r_0, \dr_0, K_{1,0})$ and
 $(r_\infty,  \dr_\infty, K_{1,\infty})$ such that for all real number $L>0$ the origin

 is a globally asymptotically stable equilibrium of the system (\ref{c55}) and (\ref{b84}) and of its homogeneous approximations.

\end{theorem}

We can then apply Corollary \ref{b36} to get an output feedback result for nonlinear systems described by~:
\begin{equation}
\label{LP2}
\dot x \;=\; \SR_n\, x\;+\; B_n\,u\;+\;\delta(t)\quad,\qquad y\;=\; x_1\ ,
\end{equation}
where $\delta~:\RR_+\rightarrow\RR^n$ is a continuous function related to the solutions as
described in the two Corollaries below
and proved in section \ref{LP17}.
Depending on wether  $\dr_0\leq\dr_\infty$ or $\dr_\infty\leq \dr_0$ we get
an output feedback result for systems in Feedback or Feedforward form.
\begin{corollary}[Feedback-form]\label{the_th}
If, in the design of $\phi_n$ and $K_1$, we select
$\dr_0\leq\dr_\infty$, then
for every positive real numbers $c_0$
and $c_\infty$ there exist a real number $L^*>0$ such that for every
$L$ in $[L^*,+\infty )$, the following holds~:\\
For every class $\mathcal{K}$ function $\gamma _z$ and
class $\KR\LR$ function
$\beta_\delta$ we can find two class $\KR\LR$ functions $\beta_x$ and $\beta_{\hat x}$,
such that, for each function
$t\in [0,T)\mapsto (x(t), \hat \XR_n(t), \delta(t),z(t))$, $T\leq +\infty $,
with $(x,\hat \XR_n)$ $C^1$ and $\delta$ and $z$ continuous, which satisfies (\ref{LP2}), (\ref{b84}),
and for $i$ in $\{1, \dots, n\}$ and $0\leq s\leq t<T$,
\\[0.7em]\null \hfill $
|z(t)|\; \leq \;
\max\left\{
\beta _\delta \Big(|z(s)|, t-s\Big)\,  ,\,
\sup_{s\leq \kappa \leq t}
\gamma _z(|x(\kappa)|)\right\}\  ,
$\hfill \null \\
$\displaystyle
|\delta _i(t)|\,\leq\, \max\left\{\vphantom{\sum_{j}^i}
\beta _\delta \Big(|z(s)|, t-s\Big),\right.
$\hfill\null\\
\refstepcounter{equation}\label{36}$(\theequation)$\null\hfill$\displaystyle \left.
\sup_{s\leq \kappa\leq t}\left\{
c_0\,  \sum_{j=1}^i |x_j(\kappa)|^\frac{1-\dr_0(n-i-1)}{1-\dr_0(n-j)}
\;+\;
c_\infty\,\sum_{j=1}^i |x_j(\kappa)|^\frac{1-\dr_\infty(n-i-1)}{1-\dr_\infty(n-j)}\right\}
\right\}\  ,
$ \\[1em]
%stopmodif
we have for all $0\leq s \leq t \leq T$~:
$$
|x(t)|\leq \beta_x(|(x(s), \hat \XR_n(s), z(s))|, t-s)\quad, \qquad
|\hat \XR_n(t)|\leq \beta_{\hat x}(|(x(s), \hat \XR_n(s), z(s))|, t-s) \ .
$$
\end{corollary}

\begin{corollary}[Feedforward form]\label{c17}
%startmodif
%
If, in the design of $\phi_n$ and $K_1$, we select
$\dr_\infty\leq\dr_0$, then
for every positive real
numbers $c_0$
and $c_\infty$, there exists a real number $L^*>0$ such that for every
$L$ in $(0,L^*]$, the following holds~:\\
For every class $\mathcal{K}$ function $\gamma _z$ and
class $\KR\LR$ function $\beta_\delta$ we can find two class $\KR\LR$ functions $\beta_x$ and $\beta_{\hat x}$,
such that, for each function
$t\in [0,T)\mapsto (x(t), \hat \XR_n(t), \delta(t),z(t))$, $T\leq +\infty $,
with $(x,\hat \XR_n)$ $C^1$ and $\delta$ and $z$ continuous, which satisfies
(\ref{LP2}), (\ref{b84}),
and for $i$ in $\{1, \dots, n\}$ and $0\leq s\leq t<T$,
\\[0.7em]\null \hfill $\displaystyle
|z(t)|\; \leq \;
\max\left\{
\beta _\delta \Big(|z(s)|, t-s\Big)\,  ,\,
\sup_{s\leq \kappa \leq t}
\gamma _z(|x(\kappa)|)\right\}\  ,
$\hfill \null \\
$\displaystyle
|\delta_i (t)|\,\leq\, \max\left\{\vphantom{\sum_{j}^i}
\beta_\delta (|z(s)|, t-s),\right.
$\hfill\null\\[-0.3em]
\refstepcounter{equation}\label{c09}$(\theequation)$\null\hfill$\displaystyle
\left.
\sup_{s\leq \kappa\leq t}\left\{
c_0\,  \sum_{j=i+2}^n |x_j(\kappa)|^\frac{1-\dr_0(n-i-1)}{1-\dr_0(n-j)}
\;+\;
c_\infty\,\sum_{j=i+2}^n |x_j(\kappa)|^\frac{1-\dr_\infty(n-i-1)}{1-\dr_\infty(n-j)}\right\}
\right\}\  ,
$ \\[1em]
then we have for all $0\leq s \leq t \leq T$~:
$$
|x(t)|\leq \beta_x(|(x(s), \hat \XR_n(s), z(s))|, t-s)\quad, \qquad
|\hat \XR_n(t)|\leq \beta_{\hat x}(|(x(s), \hat \XR_n(s), z(s))|, t-s) \ .
$$
\end{corollary}

\example Following example \ref{c53}, we can consider the case where the  $\delta _i$'s are outputs of auxiliary systems given in (\ref{c15}).
Suppose there exist $n$ positive definite and radially unbounded functions
$Z _i~:\RR^{n_i}\rightarrow \RR_+$, three class $\KR$ functions $\omega_1$,
$\omega _2$
 $\omega _3$,
and a positive real number $\epsilon$ in $(0,1)$ such that~:
$$
|\delta _i(z _i,x)|\,\leq\,\omega _1(x)+\omega _2(Z _i(z_i))
\quad,\qquad \frac{\partial Z _i}{\partial z _i}(z _i)\,g _i(z _i,x)
\;\leq\; -Z _i(z _i) \;+\; \omega _3(|x|)
%\quad,\qquad \omega _3(|z _i|)\;\leq\;Z _i(z _i)\;\leq\; \omega _4(|z _i|)
\  ,
$$
then, if there exist two real number $\dr_0$ and $\dr_\infty$
satisfying $-1<\dr_0\leq\dr_\infty< \frac{1}{n-1}$ and
\begin{equation}\label{c56}
\omega _1(x) + \omega _2\left([1+\epsilon]\,\omega _3(|x|)\right)\leq
\left(\sum_{j=1}^i |x_j|^\frac{1-{\dr}_0(n-i-1)}{1-{\dr}_0(n-j)}
+ \sum_{j=1}^i |x_j|^\frac{1-{\dr}_\infty(n-i-1)}{1-{\dr}_\infty(n-j)}\right)
\ .
\end{equation}
then Corollary (\ref{the_th}) gives $L^*>0$ such that for all
$L$ in $[L^*,+\infty )$, the output feedback
(\ref{b84}) is globally asymptotically stabilizing.
Compared to already published results (see \cite{Khalil-Saberi}
and \cite{Qian}, for instance), the novelty
is in the simultaneous presence of the terms $|x_j|^\frac{1-\dr_0(n-i-1)}{1-\dr_0(n-j)}$ and
$|x_j|^\frac{1-\dr_\infty(n-i-1)}{1-\dr_\infty(n-j)}$.

On the other hand if there exists
two real number $\dr_0$ and $\dr_\infty$ satisfying
$-1<\dr_\infty\leq\dr_0< \frac{1}{n-1}$ and
$$\omega _1(x) + \omega _2\left([1+\epsilon]\,\omega _3(|x|)\right)\;\leq\;
\left(\sum_{j=i+2}^n |x_j|^\frac{1-{\dr}_0(n-i-1)}{1-{\dr}_0(n-j)}
\;+\; \sum_{j=i+2}^n |x_j|^\frac{1-{\dr}_\infty(n-i-1)}{1-{\dr}_\infty(n-j)}\right)
\ .
$$
then Corollary (\ref{c17}) gives $L^*>0$ such that for all
$L$ in $(0,L^*]$, the output feedback (\ref{b84}) is
globally asymptotically stabilizing.

\example Consider the illustrative system (\ref{b18}).
The bound (\ref{c56}) gives the condition~:
\begin{equation}
\label{LP18}
0\; < \; q\; < \; p \; <\; 2
\  .
\end{equation}
This is almost the least conservative condition we can
obtain with the domination approach. Specifically, it is shown
in \cite{Mazenc-Praly-Dayawansa} that, when $p > 2$, there is no
stabilizing output feedback. However, when $p = 2$, (\ref{c56}) is
not satisfied although the stabilization problem is
solvable (see \cite{Mazenc-Praly-Dayawansa}).

By Corollary \ref{b64}, when (\ref{LP18}) holds,
the output feedback
$$
\renewcommand{\arraystretch}{1.5}
u\,=\, L^2\phi_2(\hat \chi_1, \hat \chi_2)\quad,\qquad\left \{\begin{array}{rcl}
\dot {\hat \chi}_1 &=&L\,  \hat \chi_{2} -L\,  q_1(\ell_1 e_1)
\  ,
\\
\dot {\hat \chi}_{2} &=&\displaystyle \frac{u}{L} - L\, q_2(\ell_2\,q_1(\ell_1e_1 ))\  ,
\\
e_1 &=&\hat \chi_1 - y
\  .
\end{array}\right.
$$
with $\ell_1$, $\ell_2$, $\phi_2$, $q_1$ and $q_2$ defined in (\ref{c48}) and
(\ref{c49})
with picking
$\dr_0$ in $(-1, q-1]$ and
$\dr_\infty $ in $[p-1,1)$,
globally asymptotically stabilizes the origin of the system
(\ref{b18}),
with $L$ is chosen sufficiently large.
Furthermore, if
$\dr_0$ is chosen strictly negative and $\dr_\infty$ strictly
positive,  by Corollary \ref{b64}, convergence  to the origin occurs
in finite time, uniformly in the initial conditions.

% \example Similarly for the system
% $$
% \dot x_1 \;=\; x_2\,+\, x_1^{m}\quad,\qquad \dot x_2 \;=\; u
% $$
% the bound (\ref{c56}) gives the condition (with $\dr_0=\dr_\infty
% =\frac{m-1}{m}$)~:
% $$
% m\; >\; \frac{1}{2}
% \  .
% $$
% And, it is shown in \cite{Coron-Rosier} that for $ 0\leq
% m<\frac{1}{2} $ there is no stabilizing continuous dynamic (and also static) state feedback.

\example \label{c45}
To illustrate the feedforward result consider the system\footnote{Recall the notation (\ref{b52}).}~:
$$
\dot x_1 = x_2 + x_3^{\frac{3}{2}} + z^{3}
\  ,\quad
\dot x_2 = x_3
\  ,\quad
\dot x_3 = u
\  ,\quad
\dot z =-z^4 + x_3
\  ,\quad
y=x_1\ .
$$
For any $\varepsilon>0$, there exists a class $\KR\LR$ function $\beta_\delta$ such that~:
$$
|z(t)|^3\,\leq\, \max\left\{\beta_\delta(|z(s)|,t-s),\quad
(1+\varepsilon )\,  \sup_{s\leq \kappa \leq t}|x_3(\kappa)|^{\frac{3}{4}}\right\}
$$
%startmodif
Therefore by letting $\delta _1=x_3^{\frac{3}{2}} + z^{3}$ we get, for all $0\leq s\leq t< T$ on the time of existence of the solutions~:
$$
|\delta _1(t)|\,\leq\, \max\left\{\beta_\delta(|z(s)|,t-s),\quad \sup_{s\leq \kappa \leq t}
(1+\varepsilon )|x_3(\kappa)|^{\frac{3}{4}}\,+\,|x_3(\kappa)|^{\frac{3}{2}}\right\}
\  .$$
This is inequality (\ref{c09})  with
$\dr_0 = -\frac{1}{2}$ and $\dr_\infty \,=\,\frac{1}{4}$.
Consequently, Corollary \ref{c17} says that it is possible to
design a globally asymptotically stabilizing output feedback.
%
%stopmodif

\subsection{Proofs of output feedback results}~
\label{LP17}
\\
\textbf{Proof of Theorem \ref{c54}~:}
\begingroup
The homogeneous in the bi-limit state feedback $\phi_n$ and
the homogeneous in the bi-limit vector field $K_1$ involved in this
feedback are obtained by following the procedures given
in Sections \ref{b69} and \ref{b99}. They are
such that the origins   is a globally asymptotically
stable equilibrium of   the systems given in (\ref{b38}) and (\ref{b37}).
To this end,
%startmodif
as in \cite{Qian},
%stopmodif
%
we write the dynamics of this system in the coordinates $\hat {\XR}_n=(\hat
\chi_1, \dots, \hat \chi_n)$ and $E_1 = (e_1, \dots, e_n)$ and in the time $\tau$ defined by~:
\begin{equation}
\label{b75}
e_{i}\;=\; \hat \chi_i \;-\;\frac{x_i}{L^{ i-1}}
\quad,\qquad \frac{d}{d\tau}\;=\;\frac{1}{L}\,  \frac{d}{dt}
\  .
\end{equation}
This yields~:
%startmodif
\begin{equation}\label{72}
\left\{
\begin{array}{rcl}
\displaystyle\frac{d}{d\tau} \; \widehat \XR_n &=&
\SR_{n}\,  \widehat \XR_{n} + B_{n}  \phi_n(\hat \XR_n)) \,+\, K_1(e_1)
\\[1em]
\displaystyle\frac{d}{d\tau} \; E_1 &=&\SR_n\,  E_1 + K_{1}(e_1)
\end{array}\right.
\end{equation}
with $E_1=(e_1,\ldots,e_n)$, $\widehat \XR_n=(\hat \chi_1,\ldots,\hat
\chi_n)$.
The right hand side of  (\ref{72}) is a vector field which is homogeneous in the bi-limit with weights $((r_0, r_0),(r_\infty, r_\infty))$.

Given
$d_U > \max_j\{r_{0,j},r_{\infty,j}\}$, by applying Theorem \ref{b26}
twice, we get
two $C^1$, proper and positive definite functions
$V~:\RR^n\rightarrow\RR_+$ and $W~:\RR^n\rightarrow\RR_+$ such that for each $i$ in $\{1, \dots, n\}$, the functions  $\frac{\partial V}{\partial x_i}$ and
$\frac{\partial W}{\partial e_i}$ are homogeneous in the bi-limit, with weights $r_0$ and $r_\infty$, degrees $d_U-r_{0,i}$ and $d_U-r_{\infty,i}$
 and approximating functions
$\frac{\partial V_{0}}{\partial \hat \chi_j}$, $\frac{\partial V_{\infty}}{\partial \hat \chi_j}$
and $\frac{\partial W_{0}}{\partial e_j}$, $\frac{\partial W_{\infty}}{\partial e_j}$.
Moreover, for all $\widehat\XR_n\,\neq\,0$~:
\begin{eqnarray}
\nonumber
\frac{\partial V}{\partial \widehat \XR}(\widehat \XR_n)
\left[\SR_{n}\,  \widehat  \XR_n + B_{n} \phi_n(\widehat  \XR_n) \right]
&<&0
\  ,
\\\label{LP5}
\frac{\partial V_{0}}{\partial \widehat \XR_n}(\widehat \XR_n)\,
\left[\SR_{n}\,  \widehat  \XR_n + B_{n}  \phi_{n,0}(\widehat  \XR_n)\right]
&<&0
\  ,
\\\nonumber
\frac{\partial V_{\infty}}{\partial \widehat \XR_n}(\widehat \XR_n)
\left[\SR_{n}\,  \widehat  \XR_n + B_{n} \phi_{n,\infty}(\widehat  \XR_n)\right]
&<&0
\ ,
\end{eqnarray}
and for all $E_1\,\neq\,0$~:
\begin{eqnarray}
\nonumber
\frac{\partial W}{\partial E_1}(E_1)\,
(\SR_n\, E_1 + K_1(e_1)) &<& 0
\  ,
\\\label{c16}
\frac{\partial W_{0}}{\partial E_1}(E_1)\,
(\SR_n\,  E_1 + K_{1,0}(e_1)) &<& 0
\  ,
\\\nonumber
\frac{\partial W_{\infty}}{\partial E_1}(E_1)\,
(\SR_n\,  E_1 + K_{1,\infty}(e_1)) &<&0
\  .
\end{eqnarray}
Consider now the Lyapunov function candidate~:
%startmodif
\begin{equation}\label{b73}
U(\hat {\XR}_n,E_1)\;=\; V(\hat {\XR}_n)\,+\,
\coef
\,W(E_1)\ ,
\end{equation}
%stopmodif
%
where $\coef $ is a positive real number to be specified.
Let~:
\begin{eqnarray*}
&\eta(\hat \XR_n, E_1)=&\displaystyle \frac{\partial V}{\partial \widehat\XR_n}(\widehat\XR_n)
\left(\SR_{n}\,  \hat  \XR_n + B_{n}  \phi_n(\hat \XR_n)\,+\, K_1(e_1)\right)\\
&\gamma(E_1) =&\displaystyle -\frac{\partial W}{\partial E_1}
(E_1)\left(\SR_n\,  E_1 \,+\, K_1(e_1)\right)\ .
\end{eqnarray*}
These two functions are continuous and  homogeneous in the bi-limit with associated triples
$\left((r_0,r_0), d_U+\dr_0,\eta_0\right)$, $\left((r_\infty,r_\infty), d_U+\dr_\infty,\eta_\infty\right)$
 and $\left((r_0,r_0), d_U+\dr_0,\gamma_0\right)$, $\left((r_\infty,r_\infty), d_U+\dr_\infty,\gamma_\infty\right)$,
where  $\gamma_0$, $\gamma_\infty$ and $\eta_0$, $\eta_\infty$ are continuous functions.
Furthermore, by (\ref{c16}), $\gamma(E_1)$ is  negative definite. Hence, by (\ref{LP5}), we have~:
%startmodif
\begin{eqnarray*}
\left\{(\hat \XR_n,E_1)\in \RR^{2n}\setminus\{0\}\: :\;
\gamma(E_1) \,=\, 0 \right\}
&\ \subseteq\ &
\left\{(\hat \XR_n,E_1)\in \RR^{2n}\: :\;
\eta(\hat \XR_n, E_1) \,<\, 0\right\}\  ,
\\
\left\{(\hat \XR_n,E_1)\in \RR^{2n}\setminus\{0\}\: :\;
\gamma_0(E_1) \,=\, 0 \right\}
&\ \subseteq\ &
\left\{(\hat \XR_n,E_1)\in \RR^{2n}\: :\;
\eta_0(\hat \XR_n, E_1) \,<\, 0\right\}\  ,
\\
\left\{(\hat \XR_n,E_1)\in \RR^{2n}\setminus\{0\}\: :\;
\gamma_\infty (E_1) \,=\, 0 \right\}
&\ \subseteq\ &
\left\{(\hat \XR_n,E_1)\in \RR^{2n}\: :\;
\eta_\infty (\hat \XR_n, E_1) \,<\, 0\right\}\  .
\end{eqnarray*}
%stopmodif
% \begin{eqnarray*}
% \left\{(\hat \XR_n,E_1)\neq (0,0)\  ,\quad \gamma(E_1) \,=\, 0 \right\}\quad
% &\Rightarrow& \quad \eta(\hat \XR_n, E_1) \,<\, 0\ ,\\
% \left\{(\hat \XR_n,E_1)\neq (0,0)\  ,\quad \gamma _0(E_1) \,=\, 0 \right\}\quad
% &\Rightarrow& \quad \eta_0(\hat \XR_n, E_1) \,<\, 0\ ,\\
% \left\{(\hat \XR_n,E_1)\neq (0,0)\  ,\quad \gamma _ \infty (E_1) \,=\, 0 \right\}\quad
% &\Rightarrow& \quad \eta_\infty(\hat \XR_n, E_1) \,<\, 0\ .
% \end{eqnarray*}
Consequently, by Lemma \ref{3}, there exists a positive real number
$\coef ^*$ such that, for all $\coef  >\coef  ^*$ and all $(\hat \XR_n,E_1)\neq (0,0)$,
the Lyapunov function $U$, defined in (\ref{b73}), satisfies~:\\[1em]
\null \quad$\displaystyle
\frac{\partial U}{\partial \hat {\XR}_n}(\hat {\XR}_n,E_1)
\left(\SR_{n}\,  \hat  \XR_n + B_{n}  \phi_n(\hat \XR_n) \,+\, K_1(e_1)\right)
$\hfill\null\\\null\hfill$\displaystyle
+\; \frac{\partial U}{\partial E_1}(\hat
{\XR}_n,E_1) (E_1)\left(\SR_n\,  E_1 \,+\, K_1(e_1)\right)\;<\;0
$\quad \null \\[1em]
and the same holds for the homogeneous approximations in the $0$-limit and in the $\infty$-limit, hence the claim.
\null \ \null\hfill$\Box$\endgroup
\par\vspace{1em}\noindent
%
%\begin{proof}
\textbf{Proof of Corollary \ref{the_th}~:}\begingroup
\
We write the dynamics of the system \ref{LP2} in the coordinates $\hat {\XR}_n$ and $E_1$ and in the time $\tau$ given in (\ref{b75}).
This yields~:
%startmodif
\begin{equation}\label{c57}
\left\{
\begin{array}{rcl}
\displaystyle\frac{d}{d\tau} \; \widehat \XR_n &=&
\SR_{n}\,  \widehat \XR_{n} + B_{n}  \phi_n(\hat \XR_n)) \,+\, K_1(e_1)
\\[1em]
\displaystyle\frac{d}{d\tau} \; E_1 &=&\SR_n\,  E_1 + K_{1}(e_1)\,+\, \mathfrak D(L)
\end{array}\right.
\end{equation}
with~:
$$
\mathfrak D(L)\;=\;
\left(\frac{\delta_1}{L},\ldots,\frac{\delta_n}{L^{n}}\right) \  .
$$
We denote the solution of this system starting from $(\widehat \XR_n(0),E_1(0))$ in $\RR^{2n}$ at time $\tau$ by $(\widehat \XR_{\tau,n}(\tau),E_{\tau,1}(\tau))$.
We have~:
\begin{equation}\label{c08}
x _i(t)\;=\; L^{i-1} \, \left(\hat \chi _{\tau,i}\left(Lt\right)\,-\,e _{\tau,i}\left(Lt\right)\right)\ .
\end{equation}
The right hand side of  (\ref{c57}) is a vector field which is homogeneous in the bi-limit with weights $((r_0, r_0),(r_\infty, r_\infty))$ for $(\widehat \XR_n,E_1)$ and $(\mathfrak r_0,\mathfrak r_\infty)$ for $\mathfrak D(L)$ where $\mathfrak r_{0,i} = r_{0,i} + \dr_0$ and $\mathfrak r_{\infty,i} = r_{\infty,i} + \dr_\infty$ for each $i$ in $\{1, \dots, n\}$.
%stopmodif
%

The time function $\tau \mapsto \delta (\frac{\tau
}{L})$ is considered as an input
and when $\mathfrak D(L)=0$, Theorem \ref{c54} implies global asymptotic stability of the origin of the
system (\ref{c57}) and of its homogeneous approximations.
To complete the proof we show that there exists $L^*$
such that the "input" $\mathfrak D(L)$
satisfies the small-gain condition (\ref{b96})
of Corollary \ref{b36} for all $L>L^*$.
% with $\mathfrak r_{0,i} = r_{0,i} + \dr_0$ and
% $\mathfrak r_{\infty,i} = r_{\infty,i} + \dr_\infty$ for $i=1, \dots, n$.
Using equations (\ref{b75}) and (\ref{c08}), assumption (\ref{36}) becomes, for all $0\leq \sigma\leq \tau< LT$, and all $i$ in $\{1, \dots, n\}~$~:
%startmodif
\\[0.7em] \null \hskip -1pt$
\frac{\left|\delta_i\left(\frac{\tau}{L}\right)\right|}{L^{i}}
\leq\max\left\{
\vphantom{\displaystyle \frac{a^{\frac{a}{a}}}{a}}\textstyle
\frac{1}{L^i}\beta _\delta \left(\left|z(\frac{\sigma }{L}))\right|,\frac{\tau-\sigma }{L}\right),
\right.
$\hfill\null\\\null\hfill$
L^{-i}\sup_{\sigma \leq \kappa\leq \tau}\left\{
c_0\sum_{j=1}^i\left |L^{(j-1)}
   (\hat \chi_{\tau,j}(\kappa) - e_{\tau,j}(\kappa))
\right|^\frac{1-\dr_0(n-i-1)}{1-\dr_0(n-j)}
\right.
$ \hfill \null \\
\refstepcounter{equation}\label{b89}$(\theequation)$\null \hfill $
\left.
\left.+ c_\infty \sum_{j=1}^i\,\left |L^{(j-1)}
(\hat \chi_{\tau,j}(\kappa) - e_{\tau,j}(\kappa))
\right|^\frac{1-\dr_\infty(n-i-1)}{1-\dr_\infty(n-j)}\right\}\; \right\}
\  .
$ \\[1em]
%stopmodif
Note that when $1\leq j\leq i\leq n$  the function
$s \mapsto \frac{1 - (n-i-1)\,s}{1 - (n-j)\,s}$ is strictly increasing,
mapping $\left(-1,\frac{1}{n-1}\right)$ in $\left(\frac{n-i}{n+1-j}, \frac{i}{j-1}\right)$.
As $\dr_0\,\leq\,\dr_\infty\,<\,\frac{1}{n-1}$, we have for all $1\leq j\leq i\leq n$~:
$$
\frac{1-\dr_0(n-i-1)}{1-\dr_0(n-j)}\;\leq\;\frac{1-\dr_\infty(n-i-1)}{1-\dr_\infty(n-j)}
\;<\;\frac{i}{j-1}\  .
$$
Hence, selecting $L\geq 1$, there exists  a real number $\epsilon>0$ such that~:
$$
L^{-\epsilon}\;\geq\;
L^{(j-1)\frac{1-\dr_\infty(n-i-1)}{1-\dr_\infty(n-j)}-i}\;\geq\;
L^{(j-1)\frac{1-\dr_0(n-i-1)}{1-\dr_0(n-j)}-i}
\  .
$$
%startmodif
This implies
\\[1em] \null \hskip -1pt$\displaystyle
\frac{\left|\delta_i\left(\frac{\tau}{L}\right)\right|}{L^{i}}
\leq\max\left\{
\frac{1}{L^i}\beta _\delta \left(\left|z(\frac{\sigma }{L})\right|,\frac{\tau-\sigma }{L}\right),
\right.
$\hfill\null\\\null\hfill\qquad$
L^{-\epsilon }\sup_{\sigma \leq \kappa\leq \tau}\left\{
c_0\sum_{j=1}^i\left |
   (\hat \chi_{\tau,j}(\kappa) - e_{\tau,j}(\kappa))
\right|^\frac{1-\dr_0(n-i-1)}{1-\dr_0(n-j)}
\right.$\hfill \null \\[-0.5em]\null \hfill $
\left.
\left.
\vphantom{\displaystyle \frac{a^{\frac{a}{a}}}{a}}\textstyle
+ c_\infty \sum_{j=1}^i\,\left |
(\hat \chi_{\tau,j}(\kappa) - e_{\tau,j}(\kappa))
\right|^\frac{1-\dr_\infty(n-i-1)}{1-\dr_\infty(n-j)}\right\}\; \right\}
\  .
$\null\\[0.5em]
%stopmodif
%
On the other hand,
the function
$$
(\widehat \XR_n,E_1)\; \mapsto\;
c_0\sum_{j=1}^i\,\left |\hat \chi_j - e_j\right|^\frac{1-\dr_0(n-i-1)}{1-\dr_0(n-j)}
+ c_\infty \sum_{j=1}^i\,\left |\hat \chi_j - e_j\right|^\frac{1-\dr_\infty (n-i-1)}{1-\dr_\infty (n-j)}
$$
is homogeneous in the bi-limit with weights
$(r_0,r_0)$  and $(r_\infty,r_\infty)$ and degrees
$1-\dr_0(n-i-1) = r_{0,i}+\dr_0$ and $1-\dr_\infty (n-i-1)= r_{\infty,i}+\dr_\infty$
(see  (\ref{b60})).
Hence, by Corollary \ref{b25},
there exists a positive real
number $c_1$ such that~:
\\[1em]$\null\quad\displaystyle
c_0\,  \sum_{j=1}^i\,\left |\hat \chi_j - e_j\right|^\frac{1-\dr_0(n-i-1)}{1-\dr_0(n-j)}
+ c_\infty \,  \sum_{j=1}^i\,\left |\hat \chi_j - e_j\right|^\frac{1-\dr_\infty (n-i-1)}{1-\dr_\infty (n-j)}
$\\[0.3em]
\refstepcounter{equation}\label{LP10}(\theequation)\null\hfill$
\leq \; c_1\,\mathfrak{H}\left(
|(\widehat \XR_n,E_1)|_{(r_0,r_0)}^{\dr_0+ r_{0,i}}
, |(\widehat \XR_n,E_1)|_{(r_\infty,r_\infty)}^{\dr_\infty+ r_{\infty,i}}
\right)\  .
\quad$\\[1em]
Hence, by Corollary \ref{b36}
%startmodif
(applied in the $\tau $ time-scale),
%stopmodif
there exists $c_G $, such that for any
$L^*$ large enough such that $ c_1 L^{*-\varepsilon }\leq c_G $, the conclusion holds.
%\end{proof}
\null \ \null\hfill$\Box$\endgroup
\par\vspace{1em}\noindent
\textbf{Proof of Corollary \ref{c17}~:}\begingroup
\
The proof is similar to the previous one with the only difference
that,
when $i$ and $j$ satisfy $3\leq  i+2\leq j\leq n$,  the function
$s \mapsto \frac{1 - (n-i-1)\,s}{1 - (n-j)\,s}$ is strictly decreasing
mapping $\left(-1,\frac{1}{n-1}\right)$ in
$\left(\frac{i}{j-1},\frac{n-i}{n+1-j}\right)$.
%startmodif
Moreover the condition $-1<\dr_\infty\leq \dr_0 <\frac{1}{n-1}$
%stopmodif
gives the inequalities
%startmodif
$$
\frac{1-\dr_\infty(n-i-1)}{1-\dr_\infty(n-j)}
\;\geq\;
\frac{1-\dr_0(n-i-1)}{1-\dr_0(n-j)}
\;>\;\frac{i}{j-1}\  .
$$
Hence (\ref{LP10}) holds and by
selecting $L<1$ we obtain the existence of a positive real number $\epsilon$ such that~:
$$
L^{\epsilon}\;\geq\;
L^{(j-1)\frac{1-\dr_0(n-i-1)}{1-\dr_0(n-j)}-i}\;\geq\;
L^{(j-1)\frac{1-\dr_\infty(n-i-1)}{1-\dr_\infty(n-j)}-i}
\  .
$$
From (\ref{c09}), this yields, for all $0\leq \sigma\leq \tau< LT$, and all $i$ in $\{1, \dots, n\}$~:~:
\\[1em] \null \hskip -1pt$\displaystyle
\frac{\left|\delta_i\left(\frac{\tau}{L}\right)\right|}{L^{i}}
\leq\max\left\{
\frac{1}{L^i}\beta _\delta \left(\left|z(\frac{\sigma }{L})\right|,\frac{\tau-\sigma }{L}\right),
\right.
$\hfill\null\\\null\hfill$
L^{\epsilon }\sup_{\sigma \leq \kappa\leq \tau}\left\{
c_0\sum_{j=i+2}^n\left |
   (\hat \chi_{\tau,j}(\kappa) - e_{\tau,j}(\kappa))
\right|^\frac{1-\dr_0(n-i-1)}{1-\dr_0(n-j)}
\right.$\hfill \null \\[-0.5em]\null \hfill $
\left.
\left.
\vphantom{\displaystyle \frac{a^{\frac{a}{a}}}{a}}\textstyle
+ c_\infty \sum_{j=i+2}^n\,\left |
(\hat \chi_{\tau,j}(\kappa) - e_{\tau,j}(\kappa))
\right|^\frac{1-\dr_\infty(n-i-1)}{1-\dr_\infty(n-j)}\right\}\; \right\}
\  .
$\quad \null \\[1em]
%stopmodif
From Corollary \ref{b36}, the result holds  for all $L^*$ small enough to satisfy
$ c_1 L^{*\varepsilon }\leq c_G $.
\null \ \null\hfill$\Box$\endgroup\par

\section{Conclusion}
We have presented two new tools that can be useful in nonlinear
control design.
The first one is introduced to formalize the notion
of homogeneous approximation valid both at the origin and
at infinity.
With this formalism we have given several novel results concerning
asymptotic stability, robustness analysis and also finite time convergence
(uniformly in the initial conditions).
The second one is a new recursive design for an observer for a chain
of integrators.
The combination of these two tools allows to obtain a new result
on stabilization by output feedback for systems whose dominant
homogeneous  in the bi-limit part is a chain of integrators.

\section*{Acknowledgments}
The second author is extremely grateful to Wilfrid Perruquetti and Emmanuel Moulay   for the many discussions he has had about the notion of homogeneity in the bi-limit.
Also, all the authors would like to thank the anonymous reviewers
%startmodif
for their comments which were extremely helpful to improve the quality of the paper.
%stopmodif

\appendix

\section{Proof of Proposition \ref{prop1}}
\label{c18}
We give only the proof in the $0$-limit case since the $\infty$-limit case is similar.
Let $C$ be an arbitrary compact subset of  $\RR^n\setminus\{0\}$ and $\epsilon $ any  strictly positive real number.
By definition of homogeneity in the $0$-limit, there exists $\lambda_1\,>\,0$ such that we have~:
$$
\left|\frac{\phi(\lambda^{r_{\phi,0}}\diamond x)}{\lambda^{d_{\phi,0}}}  -  \phi_0(x)\right |
\;\leq\; 1 \qquad,\quad \forall \; x\;\in\;C
\  ,\quad
\forall \lambda\in (0, \lambda_1]\ .
$$
Hence, as $\phi_0$ is a continuous function on $\RR^n$, for all
$\lambda$ in $(0, \lambda_1]$, the function
$x\mapsto \frac{\phi(\lambda^{r_0}\diamond\, x)}{\lambda^{d_{\phi,0}}}$
takes its values in a compact set
$C_\phi = \phi_0(C) + B_{1}$ where $B_{1}$ is the unity ball.

Now, as $\zeta_0$ is continuous on the compact subset $C_\phi$,
it is uniformly continuous, i.e. there exists $\nu > 0$ such that~:
$$
|z_1-z_2|\;<\;\nu\qquad \Longrightarrow\qquad
|\zeta_0(z_1)-\zeta_0(z_2)|<\epsilon
\  .
$$
%startmodif
Also there exists
%stopmodif
%The definition of homogeneity in the $0$-limit imply the existence of
$\mu _\epsilon>0$ satisfying~:
$$
\left|\frac{\zeta(\mu^{r_{\zeta,0}} z)}{\mu^{d_{\zeta,0}}}  -  \zeta_0(z)\right |
\;\leq\; \epsilon \quad,\qquad\forall\;z\;\in\;C_\phi
\  ,\quad \forall \mu\in (0, \mu_\epsilon]\ ,
$$
or equivalently, since $d_{\phi,0}>0$ ~:
$$
\left|\frac{\zeta(\lambda^{d_{\phi,0}} z)}{\lambda^\frac{d_{\phi,0}\,{d_{\zeta,0}}}{r_{\zeta,0}}}
 -  \zeta_0(z)\right | \;\leq\; \epsilon \quad,\qquad\forall\;z\;\in\;C_\phi
\  ,\quad
\forall \lambda \in \left(0,
 \mu_\epsilon^\frac{r_{\zeta,0}}{d_{\phi,0}}\right]\ .
$$
Similarly, there exits $\lambda _\nu $ such that~:
$$
\left|\frac{\phi(\lambda^{r_{\phi,0}}\diamond x)}{\lambda^{d_{\phi,0}}}
-  \phi_0(x)\right | \;\leq\; \nu \qquad,\quad \forall \; x\;\in\;C
\  ,\quad
\forall \lambda\in (0, \lambda_\nu]\ .
$$
It follows that~:
\begin{eqnarray*}
\left|\frac{\zeta( \phi(\lambda^{r_{\phi,0}}\diamond x))}{\lambda^\frac{d_{\phi,0}\,{d_{\zeta,0}}}{r_{\zeta,0}}}  -  \zeta_0\left(\phi_0(x)\right)\right | &\leq& \left|\frac{\zeta( \phi(\lambda^{r_{\phi,0}}\diamond x))}{\lambda^\frac{d_{\phi,0}\,{d_{\zeta,0}}}{r_{\zeta,0}}}  -  \zeta_0\left(\frac{\phi(\lambda^{r_{\phi,0}}\diamond x)}{\lambda^{d_{\phi,0}}}\right)\right | \\\nonumber
&&\qquad+\;  \left|\zeta_0\left(\frac{\phi(\lambda^{r_{\phi,0}}\diamond x)}{\lambda^{d_{\phi,0}}}\right) -  \zeta_0\left(\phi_0(x)\right)\right |\ ,\\
&\leq& 2\,\epsilon \qquad,\quad \forall \; x\;\in\;C
\  ,\quad
\forall \lambda\in \min\left\{\lambda _1,\lambda_\nu,
 \mu_\epsilon^\frac{r_{\zeta,0}}{d_{\phi,0}}\right\}\ .
\end{eqnarray*}
This establishes homogeneity in the $0$-limit of the function $\zeta\circ \phi$.

\section{Proof of Proposition \ref{b32}}
\label{c21}
We give only the proof in the $0$-limit case since the $\infty$-limit case is similar.
%startmodif
The function $\phi$ being a bijection, we can assume
without loss of generality that it is a strictly increasing function
(otherwise we take $-\phi$). This together with homogeneity in the
$0$-limit, imply that $\varphi _0$ is strictly positive.
Moreover, for each $\delta >0 $, there exists $t_{0}(\delta)\,>\,0$ such that~:
$$
\left | \frac{\phi(t)}{t^{d_0}} -  \varphi_0\right| \;\leq\; \delta\quad,\qquad\forall \; t\,\in\,(0,t_0(\delta)]\ .
$$
By letting $\lambda\;=\; \phi(t)$, this gives~:
$$
\varphi_0 - \delta  \;\leq\;\frac{\lambda}{\phi^{-1}(\lambda)^{d_0}}
\;\leq\; \varphi_0 + \delta\quad,\qquad\forall \;
\lambda\,\in\,(0,\phi (t_0(\delta))]\;,\quad \forall \,\delta \,>\,0\ .
$$
Since for $\delta<\varphi_0$ the term on the left is strictly positive, these inequalities
give~:
$$
  \left(\frac{1}{\varphi_0 + \delta}\right)^\frac{1}{d_0}\;\leq\;\frac{\phi^{-1}(\lambda)}{\lambda^\frac{1}{d_0}}  \;\leq\; \left(\frac{1}{\varphi_0 - \delta}\right)^\frac{1}{d_0}\;,\quad\forall \; \lambda\,\in\,(0,\phi^{-1}(t_0(\delta))]\,,\; \forall \delta \,\in\,(0,\varphi _0)\ .
$$
Then since the function
$\delta\,\mapsto\,\left(\frac{1}{\varphi_0 - \delta}\right)^\frac{1}{d_0}$
is continuous at zero, for every $\epsilon_1>0$, there exists $\delta_1(\epsilon_1)>0$ satisfying~:
$$
 \left(\frac{1}{\varphi_0 }\right)^\frac{1}{d_0}\;-\; \epsilon_1
\; \leq \; \left(\frac{1}{\varphi_0 + \delta_1(\epsilon_1)}\right)^\frac{1}{d_0}
\; \leq \; \left(\frac{1}{\varphi_0 - \delta_1(\epsilon_1)}\right)^\frac{1}{d_0}
\; \leq \;  \left(\frac{1}{\varphi_0 }\right)^\frac{1}{d_0}\;+\; \epsilon_1
\  .
$$
This yields~:
$$
\left | \frac{\phi^{-1}(\lambda)}{\lambda^\frac{1}{d_0}} -  \left(\frac{1}{\varphi_0}\right)^\frac{1}{d_0}\right| \;\leq\;
\epsilon_1
\qquad \forall \lambda \in (0,\lambda _-(\epsilon _1)]\  ,
$$
with $\lambda _-(\epsilon _1)=\phi (t_0(\delta_1(\epsilon_1)))$.
With a similar argument we get~:
$$
\left | \frac{\phi^{-1}(-\lambda)}{\lambda^\frac{1}{d_0}} +  \left(\frac{1}{\varphi_0}\right)^\frac{1}{d_0}\right| \;\leq\; \epsilon_1
\qquad \forall \lambda \in (0,\lambda _+(\varepsilon _1)]\ ,
$$
for some $\lambda _+>0$. Let $\lambda _0=\min\{\lambda _-,\lambda _+\}$.

Now, for $x\neq 0$ and $\lambda >0$, we have~:
$$
\left | \frac{\phi^{-1}(\lambda x)}{\lambda^\frac{1}{d_0}} -  \left(\frac{x}{\varphi_0}\right)^\frac{1}{d_0}\right|
= |x|^\frac{1}{d_0}
\left | \frac{\phi^{-1}(\lambda x)}{(x\lambda)^\frac{1}{d_0}} -  \left(\frac{1}{\varphi_0}\right)^\frac{1}{d_0}\right|
\  .$$
Therefore, for any
compact set $C$ of $\RR\setminus\{0\}$ and any $\epsilon >0$, by letting
$\epsilon _1=\frac{\epsilon}{\max_{x\in C}|x|^\frac{1}{d_0}}$,
we have~:
$$
|x|^{\frac{1}{d_0}} \epsilon _1 \; \leq \; \epsilon
\quad ,\qquad
0\; <\; |\lambda x|\; \leq\; \lambda_0\left(\epsilon_1\right)
\qquad
\forall \lambda \in \left(0, \frac{\lambda_0\left(\epsilon_1\right)}{\max_{x\in C}|x|}\right]
\; ,\  \forall x\in C
$$
and therefore~:
$$
\left | \frac{\phi^{-1}(\lambda x)}{\lambda^\frac{1}{d_0}} -  \left(\frac{x}{\varphi_0}\right)^\frac{1}{d_0}\right|
\;\leq\;\epsilon
\qquad
\forall \lambda \in \left(0, \frac{\lambda_0\left(\epsilon_1\right)}{\max_{x\in C}|x|}\right]
\; ,\  \forall x\in C\ .
$$
This establishes homogeneity in the $0$-limit of the function $\phi^{-1}$.
%stopmodif

\section{Proof of Lemma \ref{3}}
\label{c20}
The proof of this lemma  is divided into three parts.
\begin{enumerate}
\item We first show, by contradiction, that there exists a real number $c_0$
satisfying~:
$$
\eta_0(\theta) \;-\; c\,\gamma_0(\theta) \;<\; 0\qquad
\forall \theta  \in  S_{r_0}\  ,\quad
\forall c\,\geq\,c_0\ .
$$
Suppose there is no such $c_0$.
This means there is a  sequence $(\theta_i)_{i\in \NN}$ in $S_{r_0}$ which
satisfies~:
$$
\eta_0(\theta_i) \;-\; i\,\gamma_0(\theta_i) \;\geq\; 0\qquad,\quad\forall\; i\;\in\;\NN\ .
$$
The sequence $(\theta_i)_{i\in \NN}$ lives in a compact set.
Thus we can extract a convergent subsequence $(\theta_{i_\ell})_{\ell\in \NN}$
which converges to a point denoted $\theta_\infty$.

As  the functions $\eta_0$ and $\gamma_0$ are bounded on
$S_{r_0}$ and $\gamma_0$ takes non-negative values\footnote{%
%
%startmodif
Indeed, if we had $\gamma _0(x) <0$ for some $x$ in
$\RR^n\setminus\{0\}$, by letting $\epsilon = -\frac{\gamma
_0(x)}{2}$, the homogeneity in the $0$-limit of $\gamma $ would give
a real number $\lambda > 0 $ satisfying
$ \frac{\gamma(\lambda^{r_0}\diamond x)}{\lambda^{d_0}} \leq
\gamma_0(x)+ \epsilon = \frac{\gamma
_0(x)}{2} <0 $. This contradicts the fact that $\gamma$ takes nonnegative
values only. Also by continuity we have $\gamma _0(0) \geq 0$.
%stopmodif
%
},
$\gamma_0(\theta_{i_\ell})$ must go to $0$
as $i_\ell$ goes to infinity.
Since the functions $\eta_0$ and $\gamma_0$ are continuous, we
get $\gamma_0(\theta_\infty)\,=\,0$ and $\eta_0(\theta_\infty) \,\geq\, 0$,
which is impossible.
Consequently, there exist $c_0$ and $\varepsilon _0>0$ such that~:
\begin{equation}
\label{LP3}
\eta_0(\theta) \;-\; c\,\gamma_0(\theta) \;\leq \; -\varepsilon _0\;
<\; 0\qquad \forall \theta\in S_{r_0}
\  ,\quad
\forall c\,\geq\,c_0\ .
\end{equation}
Moreover, since the functions $\eta_0$ and $\gamma_0$ are homogeneous in the standard sense
(see
%startmodif
Remark \ref{LP12}), we have
%stopmodif
 the second inequality in (\ref{b53}).

Following the same argument, we can find positive real numbers $c_\infty$
and $\varepsilon _\infty $ such that~:
\begin{equation}
\label{LP4}
\eta_\infty(\theta) \;-\; c\,\gamma_\infty(\theta) \;<\; -\varepsilon _\infty
\qquad \forall \theta\in S_{r_\infty}
\  ,\quad
\forall c\,\geq\,c_\infty
\end{equation}
%startmodif
and the third inequality in (\ref{b53}) holds.
%stopmodif

In the following,  let~:
$$
c_1\;=\; \max\{c_0,c_\infty \}
\quad ,\qquad
\varepsilon _1\;=\; \min\{\varepsilon _0,\varepsilon _\infty \}
\  .
$$
\item
Since $\eta$ and $\gamma $ are
homogeneous in the $0 $-limit,
there exists $\lambda _0$ such that,
for all $\lambda \in (0,\lambda_{0}]$
and all $\theta\,\in\, S_{r_0}$,  we have~:
$$
\eta(\lambda^{r_0} \diamond \theta)
\; \leq \;
\lambda^{d_0}\,  \eta _0(\theta )
\;+\;
\lambda^{d_0}\, \frac{\varepsilon _1}{4}
\  ,\quad
\lambda^{d_0}\,  \gamma  _0(\theta )
\;-\;
\lambda^{d_0}\, \frac{\varepsilon _1}{4c_1}
\; \leq \;
\gamma (\lambda^{r_0} \diamond \theta)
\  ,
$$
which gives readily
$$
\eta(\lambda^{r_0} \diamond \theta) \;-\; c_1\,\gamma(\lambda^{r_0}
\diamond \theta) \; \leq  \;
\lambda^{d_0}\,  \eta _0(\theta )
\;+\;
\lambda^{d_0}\, \frac{\varepsilon _1}{2} - c_1\lambda^{d_0}\,  \gamma  _0(\theta )
\ .
$$
Using (\ref{LP3}), we get
%startmodif
$$
\eta(\lambda^{r_0} \diamond \theta) \;-\; c_1\,\gamma(\lambda^{r_0}
\diamond \theta) \; \leq \;
-\lambda^{d_0}\, \frac{\varepsilon _1}{2}
\qquad
\forall \lambda \in (0,\lambda_{0}]\; ,\  \forall\theta\,\in\,
S_{r_0}
\ ,
$$
and therefore, since
$\gamma $ takes non negative values,
%stopmodif
% $$
% \eta(\lambda^{r_0} \diamond \theta)  - c\,\gamma(\lambda^{r_0}
% \diamond \theta)  \leq  -\lambda ^{d_0}\frac{\varepsilon _1}{2}
% - (c-c_1)\,  \gamma(\lambda^{r_0} \diamond \theta)
% \quad
% \forall \lambda \in (0,\lambda_{0}]\; ,\  \forall\theta\,\in\,
% S_{r_0}
% \ .
% $$
% Finally, the function $\gamma $ taking non negative values, we get~:
$$
\eta(\lambda^{r_0} \diamond \theta)  - c\,\gamma(\lambda^{r_0}
\diamond \theta) \,\leq \, -\lambda ^{d_0}\frac{\varepsilon _1}{2} \ ,
\quad
\forall \lambda \in (0,\lambda_{0}]\; ,\  \forall\theta\,\in\,
S_{r_0}
\; ,\
\forall c\geq c_1
\ .
$$
Similarly,
there exists $\lambda_\infty $ satisfying~:
$$
\eta(\lambda^{r_\infty} \diamond \theta)  - c\,\gamma(\lambda^{r_\infty}
\diamond \theta) \,\leq \, -\lambda ^{d_\infty}\frac{\varepsilon _1}{2}
\quad
\forall \lambda \in [\lambda_\infty,+\infty)\; ,\  \forall\theta\,\in\,
S_{r_\infty }
\; ,\
\forall c\geq c_1
\  .
$$
Consequently, for each $c\,\geq \,c_1$,  the set
$$\{x \,\in\,\RR^n \setminus\{0\}\;|\; \eta(x) \;-\; c\,\gamma(x) \;\geq\; 0 \} \ ,$$
if not empty, must be a subset of
$$
C \;=\; \left\{x \,\in\,\RR^n \;:\;    |x|_{r_0}\;\geq\;
\lambda_0 \right\}
\; \bigcup\;
\left\{x \,\in\,\RR^n \;:\;  |x|_{r_\infty}\;\leq\; \lambda_\infty \right\}
\  .
$$
which is compact and does not contain the origin.

\item Suppose now that for all $c$ the first inequality in (\ref{b53}) is not true,
this means that, for all integer $c$ larger then $c_1$
there exists $x_c$ in $\RR^n$ satisfying~:
$$
\eta(x_c) \;-\; c\,\gamma(x_c) \;\geq\; 0
$$
and therefore $x_c$ is in $C$. Since $C$ is a compact set, there is a
convergent subsequence $(x_{c_\ell})_{\ell\in\NN}$ which converges to a
point denoted $x^*$ different from zero. And as above, we must have
$\gamma (x^*)=0$ and $\eta (x^*)\geq 0$. But this contradicts the
assumption, namely
$$
\left\{\  x\in \RR^n\setminus\{0\}\  ,\quad \gamma (x) = 0 \  \right\} \qquad
\Rightarrow \qquad \eta(x) < 0\ .
$$
\end{enumerate}

\section{Proof of Proposition \ref{prop2}}
\label{c14}
Because the vector field $f$ is homogeneous in the $\infty$-limit, its approximating
vector field $f_\infty$ is homogeneous in the standard sense (see Remark \ref{LP12}).
Let $d_{V_\infty}$ be a positive real number larger than $r_{\infty,i}$, for all $i$ in $\{1, \dots, n\}$.
Following Rosier \cite{Rosier}, there exists a $C^1$, positive
definite, proper and
homogeneous function $V_\infty~:\RR^n\rightarrow\RR_+$, with weight
$r_\infty$ and degree $d_{V_\infty}$, satisfying~:
\begin{equation}\label{b17}
\frac{\partial V_\infty}{\partial x}(x)f_\infty(x) \;<\; 0\qquad,\quad \forall \;x\;\neq\;0\ .
\end{equation}
% MODIF FAITE
From Point P1 in Section \ref{b55}, we know that the function $x\mapsto\frac{\partial V_\infty}{\partial x}(x)f(x)$
is homogeneous  in the $\infty$-limit with associated triple $\left
(r_\infty,\dr_\infty+d_{V_\infty},\frac{\partial V_\infty}{\partial
x}(x)f_\infty(x) \right)$.
Let
$$
\epsilon_\infty\;=\; -\frac{1}{2}\max_{\theta\,\in\,S_{r_\infty}} \left\{ \frac{\partial V_\infty}{\partial x}(\theta)f_\infty(\theta)\right\}\ ,
$$
and note that, by inequality (\ref{b17}), $\epsilon_\infty$ is a strictly positive real number.
By definition of homogeneity  in the $\infty$-limit, there exists
$\lambda_\infty$ such that~:
$$
\left|\frac{\frac{\partial V_\infty}{\partial x}(\lambda^{r_\infty}\diamond \theta)f(\lambda^{r_\infty}\diamond \theta)}{\lambda^{d_{V_\infty}+\dr_\infty} } -
\frac{\partial V_\infty}{\partial x}(\theta)f_\infty(\theta)\right|
\;\leq\;\epsilon_\infty
\qquad \forall \theta\in S_{r_\infty}
\; ,\
\forall \lambda \geq \lambda_\infty
\ .
$$
This yields~:
\begin{eqnarray*}
\frac{\partial V_\infty}{\partial x}(\lambda^{r_\infty}\diamond \theta)f(\lambda^{r_\infty}\diamond \theta)
&\leq&  \lambda^{d_{V_\infty}+\dr_\infty} \left(\frac{\partial V_\infty}{\partial x}(\theta)f_\infty(\theta)\,+\,\epsilon_\infty\right)\ ,\\
&\leq&  -\,\lambda^{d_{V_\infty}+\dr_\infty} \, \epsilon_\infty
\qquad \forall \theta\in S_{r_\infty}
\; ,\
\forall \lambda \geq \lambda_\infty
\end{eqnarray*}
or in other words~:
\begin{equation}\label{b15}
\frac{\partial V_\infty}{\partial x}(x)\,f(x)\,<\,0\quad ,\qquad
\qquad \forall \,  x\: : |x|_{r_\infty}\,\geq\,\lambda_\infty\ .
\end{equation}
This establishes global asymptotic stability of the compact set~:
$$
\CR_\infty\;=\;\{x\,:\,V_\infty(x)\;\leq\;v_\infty\}\ ,
$$
where $v_\infty $ is given by~:
$$
v_\infty\;=\;\max_{|x|_{r_\infty}\,=\,\lambda_\infty} \{ V_{\infty}(x)\}
\  .
$$

\section{Proof of Theorem \ref{b26}}
\label{c13}
The proof is divided in three steps.
First, we define three Lyapunov functions $V_0$, $V_m$ and $V_\infty$.
Then we build another Lyapunov function $V$ from these three ones.
Finally we show that its derivative along the trajectories of the system
(\ref{c19}) and its homogeneous approximations are negative definite.
\begin{enumerate}
\item
As established in the proof of Proposition \ref{prop2}, there exist
a positive real number $\lambda_\infty$ and a $C^1$ positive definite,
proper and homogeneous function
$V_\infty~:\RR^n\rightarrow\RR_+$, with weight $r_\infty$ and degree
$d_{V_\infty}$ satisfying (\ref{b15}).
Similarly, there exist a number $\lambda_0>0$ and a $C^1$ positive definite, proper and
homogeneous function $V_0~:\RR^n\rightarrow\RR_+$, with weight $r_0$ and degree $d_{V_0}$,
satisfying~:
\begin{equation}\label{b22}
\frac{\partial V_0}{\partial x}(x)\,f(x)\,<\,0\quad ,\qquad
\forall\,x\;:\;0\,  <\,  |x|_{r_0}\,\leq\,\lambda_0 \ .
\end{equation}
Finally, global asymptotic stability of the origin of the system $\dot x = f(x)$ implies the
existence of a $C^1$, positive definite  and proper  function $ V_m~:\RR^n\rightarrow\RR_+$
satisfying~:
\begin{equation}\label{b50}
\frac{\partial V_m}{\partial x}(x)\,f(x)\;<\;0 \quad, \qquad \forall \, x\,\neq\,0\ .
\end{equation}
\item
Now we build a function $V$ from the functions
$V_m$, $V_\infty$ and $V_0$. For this, we follow a technique used by
Mazenc in \cite{Mazenc}
(see also \cite{liu-chitour-sontag}).
Let $v_\infty$  and $v_0$ be two strictly positive real numbers such that
$v_0<v_\infty $ and
$$
v_\infty\;\geq\;\max _{x:\,  |x|_{r_\infty}\,\leq\,\lambda_\infty} V_m(x)
\quad,\qquad v_0\;\leq\;\min _{x:\,  |x|_{r_0}\,\geq\,\lambda_0} V_m(x)
\ .
$$
This implies~:
%startmodif
\begin{eqnarray*}
\{x\in\RR^n\: :\; V_m(x)\,\geq\, v_\infty\}&\; \subseteq\; &
\{x\in\RR^n\: :\; |x|_{r_\infty}\,\geq\,\lambda_\infty\}
\  ,
\\
\{x\in\RR^n\: :\; V_m(x)\,\leq\, v_0\}&\; \subseteq\; &
\{x\in\RR^n\: :\; |x|_{r_0}\,\leq\,\lambda_0\}
\  .
\end{eqnarray*}
%stopmodif
Let $\omega_0$ and $\omega_\infty$ be defined as~:
$$
\omega_0\;=\;\min_{x\,:\,\frac{1}{2}\,v_0 \leq V_m (x)\leq v_0 }\frac{V_m(x)}{V_0 (x)}
\quad,\qquad
\omega_\infty \;=\;\max_{x\,:\,v_\infty \leq V_m (x)\leq 2\,v_\infty }\frac{V_m(x)}{V_\infty (x)}
\ .
$$
We have~:
\begin{eqnarray*}
\omega_\infty \,V_\infty (x)\;-\;V_m(x) &\geq &0 \qquad,\quad \forall \,x\;:\; v_\infty \leq V_m (x)\leq 2\,v_\infty\ , \\
V_m(x) \;-\; \omega_0 \,V_0 (x)  &\geq & 0 \qquad,\quad \forall \,x\;:\;\,\frac{1}{2}\,v_0 \leq V_m (x)\leq v_0
\ .
\end{eqnarray*}

Let~:
\\[1em]\noindent$\displaystyle
V(x) \; = \; \omega_\infty \, \varphi_\infty (V_m (x))V_\infty(x) +$\hfill\null\\[1em]
\null\hfill$
\left[1 - \varphi_\infty (V_m(x))\right]\varphi_0 (V_m (x))\,V_m(x) + \omega_0 \,\left[1- \varphi_0 (V_m (x))\right] V_0(x)
$\\[1em]
where $\varphi_0 $ and $\varphi_\infty $ are $C^1$ non decreasing functions satisfying~:
\begin{eqnarray}\label{b23}
\varphi_0 (s) \; = \; 0 \qquad \forall \, s \leq \frac{1}{2}\,v_0
\quad & , &\qquad
\varphi_0 (s) \; = \; 1 \qquad \forall \, s \geq  v_0  \ .
\\\label{b24}
\varphi_\infty (s) \; = \; 0 \qquad \forall \, s \leq v_\infty
\quad& , &\qquad
\varphi_\infty (s) \; = \; 1 \qquad \forall \, s \geq 2 v_\infty  \ .
\end{eqnarray}
Then  $V$ is  $C^1$, positive definite and proper.
Moreover, by construction~:
$$
V (x)
=\left\{
\begin{array}{l@{\hskip -2cm}l}
\omega_0\,V_0(x)
   & \forall x\,:\, V_m(x)\,\leq\,\frac{1}{2}v_0\ ,
\\[0.5em]
\varphi_0 (V_m (x))\,V_m(x)+\omega_0 \,\left[1- \varphi_0 (V_m (x))\right] V_0(x)
\\
   & \forall x\,:\, \frac{1}{2}v_0\,\leq\,V_m(x)\,\leq\,v_0\ ,
\\[0.5em]
V_m(x)
   & \forall x\,:\, v_0\,\leq\,V_m(x)\,\leq\,v_\infty\ ,
\\[0.5em]
\omega_\infty \, \varphi_\infty (V_m (x))V_\infty(x)
+ \left[1 - \varphi_\infty (V_m(x))\right]\,V_m(x)
\\
   & \forall x\,:\, v_\infty\,\leq\,V_m(x)\,\leq\,2\,v_\infty\ ,
\\[0.5em]
\omega_\infty\,V_\infty(x)
   & \forall x\,:\, V_m(x)\,\geq\,2\,v_\infty\ .
\end{array}\right.
$$

Thus for each $i$ in $\{1, \dots, n\}$~:
\begin{equation}\label{c47}
\frac{\partial V}{\partial x_i} (x) \;=\; \omega_\infty\,\frac{\partial V_\infty}{\partial x_i}(x)\quad,\qquad \forall x\,:\, V_m(x)\,>\,2\,v_\infty\ ,
\end{equation}
and
\begin{equation}\label{c46}
\frac{\partial V}{\partial x_i} (x) \;=\; \omega_0\,\frac{\partial V_0}{\partial x_i}(x)\quad,\qquad \forall x\,:\, V_m(x)\,<\,\frac{1}{2}v_0\ .
\end{equation}
%startmodif
Since $\frac{\partial V_\infty}{\partial x_i}$ and $\frac{\partial
V_0}{\partial x_i}$ are homogeneous in the standard sense, this
proves that for each $i$ in $\{1, \dots, n\}$,  $\frac{\partial V}{\partial x_i}$
is homogeneous in the bi-limit, with weights $r_0$ and $r_\infty $ and degrees $d_{V_0}-r_{0,i}$ and
$d_{V_\infty}-r_{\infty,i}$.
\item It remains to show that the Lie derivative of $V$ along $f$ is negative definite.
To this end note that,  for all $x$ such that
$\frac{1}{2}v_0\,\leq\,V_m(x)\,\leq\,v_0$,
\\[1em]$\displaystyle
\frac{\partial V}{\partial x} (x)f(x)
\;=\;\varphi_0'(V_m (x))\,[V_m(x) - \omega_0 \,V_0(x)]
\frac{\partial V_m}{\partial x}(x) f(x)
$\hfill\null\\%[1em]
$\null\hfill\displaystyle
+ \omega_0\left[1- \varphi_0 (V_m (x))\right]
\frac{\partial V_0}{\partial x} (x) f(x)+
\varphi_0(V_m (x))\frac{\partial V_m}{\partial x} (x)f(x)
$\\[1em]
and, for all $x$ such that $v_\infty\,\leq\,V_m(x)\,\leq\,2\,v_\infty$,
\\[1em]$\displaystyle
\frac{\partial V}{\partial x} (x)f(x)
\;=\;\varphi_\infty'(V_m (x))\,[\omega_\infty \,V_\infty(x) - V_m(x)]
\frac{\partial V_m}{\partial x} (x) f(x)
$\hfill\null\\%[1em]
$\null\hfill\displaystyle
+ \omega_\infty\,\varphi_\infty (V_m (x))\frac{\partial V_\infty}{\partial x}(x) f(x)
+  \left[1- \varphi_\infty (V_m (x))\right]\frac{\partial V_m}{\partial x} (x) f(x)
$\\[1em]
By (\ref{b15}), (\ref{b22}), (\ref{b50}), (\ref{b23}) and
(\ref{b24}), these inequalities  imply~:
$$
\frac{\partial V}{\partial x}(x)\,f(x) \;<\;0 \quad,\qquad \forall\, x\;\neq\;0\ .
$$
which proves the claim.
\end{enumerate}

\section{Proof of Corollary \ref{b51}}
\label{c03}
Recall equation (\ref{c41}) and consider the functions $\eta _1~:\RR^n\times\RR^m\rightarrow \RR$ and
$\gamma _1~:\RR^n\times\RR^m\rightarrow \RR_+$ defined as~:
$$
\eta _1(x,\delta)\,  =\,   \frac{\partial V}{\partial x}(x)\,[f(x,\,\delta) - \frac{1}{2}\,f(x,\,0)]
\  ,\quad
\gamma _1(x,\delta)
\,  =\,     \sum_{j=1}^m
\mathfrak{H}\left(|\delta_j|^\frac{d_{V_0}+\dr_0}{\mathfrak{r}_{0,j}}, |\delta_j|^\frac{d_{V_\infty}+\dr_\infty}{\mathfrak{r}_{\infty,j}}\right)
\: .
$$
These functions are homogeneous in the bi-limit with weights $r_{0}$ and
$r_{\infty}$ for $x$ and
$\mathfrak{r}_{0}$ and $\mathfrak{r}_{\infty}$ for $\delta$ and degrees $d_{V_0} + \dr_0$ and $d_{V_\infty} + \dr_\infty$.
Since the function $x\mapsto \frac{\partial V}{\partial x}(x)\,f(x,0) $ is negative definite, then~:
$$
\{(x,\delta) \in \RR^{n+m}\setminus\{0\}\: :\;  \gamma _1(x,\delta)\;=\;0\}
\quad \subseteq\quad \{(x,\delta) \in \RR^{n+m}\: :\; \eta _1(x,\delta)\;<\;0\}\ .
$$
Moreover, since the homogeneous approximations
%startmodif
of $\eta $
%stopmodif
is negative definite, we get~:
\begin{eqnarray*}
\{(x,\delta) \in \RR^{n+m}\setminus\{0\}\: :\; \gamma
_{1,0}(x,\delta)\;=\;0\}&\quad  \subseteq\quad  &
\{(x,\delta) \in \RR^{n+m}\: :\;
\eta _{1,0}(x,\delta)\;<\;0\}\\
\{(x,\delta) \in \RR^{n+m}\setminus\{0\}\: :\;  \gamma
_{1,\infty}(x,\delta)\;=\;0\}&\quad  \subseteq\quad  &
\{(x,\delta) \in \RR^{n+m}\: :\;
\eta _{1,\infty}(x,\delta)\;<\;0\}
\end{eqnarray*}
Hence, by Lemma \ref{3}, there exists a positive real number
 $c_\delta$ such that~:
\begin{equation}\label{c51}
\frac{\partial V}{\partial x}(x)\,\left[f(x,\,\delta) - \frac{1}{2}\,f(x,\,0)\right] \;\leq\;
c_\delta\,\sum_{j=1}^m\,
\mathfrak{H}\left(|\delta_j|^\frac{d_{V_0}+\dr_0}{\mathfrak{r}_{0,j}}, |\delta_j|^\frac{d_{V_\infty}+\dr_\infty}{\mathfrak{r}_{\infty,j}}\right)\  .
\end{equation}
Consider now the functions $\eta _2~:\RR^n\rightarrow\RR_+$ and $\gamma _2~:\RR^n\rightarrow\RR_+$ defined as~:
$$\eta _2(x)\,=\,
\mathfrak{H}\left(V(x)^{\frac{d_{V_0}+\dr_0}{d_{V_0}}}, V(x)^\frac{d_{V_\infty}+\dr_\infty}{d_{V_\infty}}\right)
\quad,\qquad
\gamma _2(x)=-\frac{1}{2}\,\frac{\partial V}{\partial x}(x)\,f(x,0)\ .$$
They are homogeneous in the bi-limit with weights $r_{0}$ and
$r_{\infty}$ and degrees $d_{V_0}+\dr_0$ and $d_{V_\infty}+\dr_\infty$.
Since $\gamma _2$ and its homogeneous approximations are positive definite, by Corollary \ref{b25},
there exists a positive real number
$c_V$ such that~:
\begin{equation}\label{c52}
\frac{1}{2}\,\frac{\partial V}{\partial x}(x)\,f(x,0) \;\leq\;
-c_V\,\mathfrak{H}\left(V(x)^{\frac{d_{V_0}+\dr_0}{d_{V_0}}}, V(x)^\frac{d_{V_\infty}+\dr_\infty}{d_{V_\infty}}\right)\ .
\end{equation}
The two inequalities (\ref{c51}) and (\ref{c52}) yield the claim.

\section{Proof of Corollary \ref{b36}}
\label{c04}
Let $d_{V_0}$ and $d_{V_\infty}$ be such that the assumption of Theorem \ref{b26} holds.
For each $i$ in $\{1,\ldots,m\}$,
let $\mu _i~:\RR_+\rightarrow\RR_+$ be the strictly increasing function
defined as (see  (\ref{c41}))~:
%startmodif
\begin{equation}
\label{LP15}
\mu_i(s)\;=\;
\mathfrak{H}\left(s^{q_i},s^{p_i}
\right)
\  ,
\end{equation}
where~:
$$
p_i=\frac{\dr_\infty+d_{V_\infty}}{\mathfrak{r}_{\infty,i}}
\quad ,\qquad
q_i=\frac{\dr_0+d_{V_0}}{\mathfrak{r}_{0,i}}
\  .
$$
%stopmodif
%The function $\mu _i$ is homogeneous in the bi-limit with approximating functions $|s|^\frac{d_0+d_{V_0}}{\mathfrak{r}_{0,i}}$  and $|s|^\frac{d_\infty+d_{V_\infty}}{\mathfrak{r}_{\infty,i}}$.
%
We first prove that
the inequality given by  Corollary \ref{b51}\  implies
%startmodif
that the system (\ref{c26}), with $\delta $ as input and $x$ as
output is input-to-state stable with
a linear
gain between $\sum_{i=1}^m\mu_i(|\delta _i|)$ and
$\mathfrak{H}\left(|x|_{r_0}^{\dr_0+d_{V_0}},|x|_{r_\infty}^{\dr_\infty+d_{V_\infty}}\right)$.
%stopmodif
To do so we introduce the
function $\alpha~:\RR_+\rightarrow\RR_+$ as~:
$$
\alpha(s)\;=\;\mathfrak{H}\left(s^\frac{\dr_0+d_{V_0}}{d_{V_0}},s^\frac{\dr_\infty+d_{V_\infty}}{d_{V_\infty}}\right)\quad,
\qquad s\,\geq\, 0
\ .
$$
This function is a bijection, strictly increasing, and homogeneous in the bi-limit with approximating functions $s^\frac{d_{V_0}+\dr_0}{d_{V_0}}$
and $s^\frac{d_{V_\infty}+\dr_\infty}{d_{V_\infty}}$.
%startmodif
Moreover, from Proposition \ref{prop1}, the function $x\mapsto \alpha(V(x))$ is
 positive definite and homogeneous in the bi-limit with associated weights $r_0$ and $r_\infty$ and degrees $\dr_0+d_{V_0}$ and $\dr_\infty+d_{V_\infty}$.
Moreover its approximating homogeneous functions $V_0(x)^\frac{d_{V_0}+\dr_0}{d_{V_0}}$ and $V_\infty(x)^\frac{d_{V_\infty}+\dr_\infty}{d_{V_\infty}}$ are positive definite as well.
Hence, we get from Corollary \ref{b25} the existence of a positive real number $c_1$ satisfying~:
\begin{equation}
\label{LP14}
\mathfrak{H}\left(|x|_{r_0}^{\dr_0+d_{V_0}},|x|_{r_\infty}^{\dr_\infty+d_{V_\infty}}\right) \;\leq\;
c_1\, \alpha(V(x))\quad,\qquad \forall \;x\;\in\;\RR^n\ .
\end{equation}
On the other hand,
%stopmodif
from inequality (\ref{b98}) in Corollary \ref{b51}, we have
the property~:\\[0.5em]
$\displaystyle
\left\{(x,\delta)\,\in\,\RR^n\times\RR^m\: :\;
\alpha (V(x))\; \geq \; 2\,  \frac{c_\delta }{c_V}\,  \sum_{i=1}^m\mu_i(|\delta _i|)\right\}
$\\[0.1em]
\refstepcounter{equation}\label{LP13}(\theequation)
\null\hfill$\displaystyle\subseteq \left\{(x,\delta)\,\in\,\RR^n\times\RR^m\: :\;
\frac{\partial V}{\partial x}(x)\,f(x,\delta)
\; \leq\;   -\frac{c _V}{2}\,\alpha(V(x))\right\}\ .
$
\par\vspace{1em}\noindent
%startmodif
In the following, let
$t\in [0,T)\mapsto (x(t),\delta (t),z(t))$, be any
function which satisfies (\ref{c26}) on $[0,T)$
and (\ref{LP11}) and (\ref{b96}) for all $0\leq s\leq t\leq T$.
From \cite{Sontag-Wang}, we know the inclusion (\ref{LP13}) implies the
existence of
a class
$\KR\LR$ function $\beta _V$ such that, for all $0\leq s\leq
t\leq T$,
%stopmodif
\begin{equation}\label{b95}\null \qquad
V(x(t)) \,\leq\, \max\left\{
\beta _V(V(x(s)),t-s)\,  ,\,  \sup_{s\leq \kappa \leq t}
\left\{\alpha^{-1}\left(\frac{2c_\delta }{c_V}\sum_{j=1}^m\, \mu _j(|\delta_j(\kappa)|)
\right)\right\}\right\}
\end{equation}
With $\alpha$ acting on both sides of inequality (\ref{b95}),
(\ref{LP14}) gives,
for all $0\leq s\leq t\leq T$,
\\[0.5em]
$\displaystyle
\mathfrak{H}\left(|x(t)|_{r_0}^{\dr_0+d_{V_0}},|x(t)|_{r_\infty}^{\dr_\infty+d_{V_\infty}}\right)
$\hfill\null\\[0.3em]
\refstepcounter{equation}(\theequation)\label{c43}\null\hfill$\displaystyle
\leq\;\max\left\{
c _1\,\alpha\circ\beta _V(V(x(s)),t-s)\,  ,\,  \frac{2c_1c_\delta }{c_V}\,
\sup_{s\leq \kappa \leq t}
\left\{\sum_{j=1}^m\, \mu _j(|\delta_j(\kappa)|)\right\}
\right\}\  .
$\\[0.5em]
This is the linear gain property required.
%startmodif
To conclude the proof it remains to show the existence of $c_G$ such that a small gain
property is satisfied.
%stopmodif

First,  note that the function $x\mapsto \mathfrak{H}\left(|x|_{r_0}^{\dr_0+d_{V_0}},|x|_{r_\infty}^{\dr_\infty+d_{V_\infty}}\right)$ is positive definite and homogeneous in the bi-limit with weights $r_0$ and $r_\infty$,
and degrees $\dr_0+d_{V_0}$ and $\dr_\infty+d_{V_\infty}$.
By Proposition \ref{prop1}, for $i$ in $\{1, \dots, m\}$ the same holds with the function
$x\mapsto \mu _i\left(\mathfrak{H}\left(|x|_{r_0}^{\mathfrak{r}_{0,i}},|x|_{r_\infty}^{\mathfrak{r}_{\infty,i}}\right)
\right)$.
Hence, by Corollary \ref{b25}, there exists a positive real number $c_2$
satisfying~:
\begin{equation}\label{c42}
\mu _i\left(
\mathfrak{H}\left(|x|_{r_0}^{\mathfrak{r}_{0,i}},|x|_{r_\infty}^{\mathfrak{r}_{\infty,i}}\right)
\right)\;\leq\;
c_2\,\mathfrak{H}\left(|x|_{r_0}^{\dr_0+d_{V_0}},|x|_{r_\infty}^{\dr_\infty+d_{V_\infty}}\right)
\quad \forall \;x \in\RR^n\ .
\end{equation}
Let $C_i$ for $i$ in $\{1, \dots, m\}$ be the class $\KR _\infty$ functions defined as
$$
C_i(c)=\max\{c^{q_i}, c^{p_i}\} +c^\frac{p_iq_i}{q_i+p_i}+c^{p_i+q_i}
\  .
$$
%startmodif
From (\ref{LP15}), we get, for each $s>0$ and $c>0$,
$$
\frac{\mu_i(c s)}{\mu_i(s)} \;=\; c^{q_i} \frac{(1+s^{q_i})(1+c^{p_i} s^{p_i})}{(1+s^{p_i})(1+c^{q_i}s^{q_i})}\\
\;\leq\; c^{q_i}
\left[\frac{1+c^{p_i} s^{p_i+q_i}}{1+c^{q_i} s^{p_i+q_i}}
+ \frac{s^{q_i}}{1+c^{q_i}s^{q_i+p_i}}
+\frac{c^{p_i} s^{p_i}}{1+s^{p_i}}\right]\ .
$$
where~:
$$
c^{q_i}
\frac{1+c^{p_i} s^{p_i+q_i}}{1+c^{q_i} s^{p_i+q_i}}
\leq \max\{c^{q_i}, c^{{p_i}}\}
\  ,\quad
\frac{c^{q_i}s^{q_i}}{1+c^{q_i}s^{q_i+p_i}}
\leq c^\frac{p_iq_i}{q_i+p_i}
\  ,\quad
\frac{c^{q_i}c^{p_i} s^{p_i}}{1+s^{p_i}}
\leq c^{p_i+q_i}
\  .
$$
Hence, by continuity at $0$, we have~:
\begin{equation}\label{c40}
\mu_i(c\,s)\;\leq\; C_i(c)\,\mu_i(s)
\qquad \forall (c,s)\in \RR_+^2\ .
\end{equation}
Consider the positive real numbers $c_1$, $c_2$, $c_\delta$ and $c_V$
previously introduced, and select $c_G$ in $\RR_+$ satisfying~:
%stopmodif
\begin{equation}\label{c39}
c _G\; <\; \min_{1\leq i\leq m}C _i^{-1}\left(\frac{c_V}{2\,m\,c_1\,c_2\,c_\delta}\right)
\ .
\end{equation}
%where for $i$ in $\{1, \dots, m\}$  the function $C_i^{-1}$ is the class
%$\KR _\infty$ left inverse functions of $C_i$.
To show that such a selection for $c_G$ is appropriate, observe
that by
%stopmodif
(\ref{c42}) and (\ref{c40}) and $\mu _i$ acting on both sides of the inequality
(\ref{b96}), we get for each $i$ in $\{1, \dots, m\}$ and
%startmodif
all $0\leq s \leq t \leq T$,
%stopmodif
\\[0.5em]
$\displaystyle
\mu _i(|\delta _i(t)|)
\,\leq\, \max\Big\{\mu_i\circ\beta_\delta(|z(s)|, t-s)\,  ,\,
$\hfill\null\\\null\hfill$\displaystyle
C _i(c _G)\,c_2\; \sup_{s\leq \kappa \leq t}  \left\{\,\mathfrak{H}\left(|x(\kappa)|_{r_0}^{\dr_0+d_{V_0}},|x(\kappa)|_{r_\infty}^{\dr_\infty+d_{V_\infty}}\right)
\right\}\Big\}\ .
$\\[0.5em]
Consequently~:
\\[0.5em]
$\displaystyle
\sum_{i=1}^m\mu_i(|\delta _i(t)|)
\,\leq\,  \max\Big\{m\,  \max_{1\leq i \leq
m}\{\mu_i\circ\beta_\delta (|z(s)|, t-s)\}\,  ,\,
$\hfill\null\\
\refstepcounter{equation}\label{LP16}(\theequation)\null\hfill$
\left(m\max_{1\leq i\leq m}C_i(c _G)\,c_2\right)\,\sup_{s\leq \kappa \leq t}  \left\{\,\mathfrak{H}\left(|x(\kappa)|_{r_0}^{\dr_0+d_{V_0}},|x(\kappa)|_{r_\infty}^{\dr_\infty+d_{V_\infty}}\right)\right\}\Big\}
$\ .\\[0.5em]
%startmodif
Since (\ref{c39}) yields~:
$$
\frac{2c_1c_\delta }{c_V}\,  m\max_{1\leq i\leq m}C_i(c _G)\,c_2\; <\; 1
\ ,
$$
the existence of the function $\beta_x$ follows
from (\ref{LP11}), (\ref{c43}), (\ref{LP16})
and the (proof of the) small gain
Theorem \cite{Jiang-Teel-Praly}.
%stopmodif

\section{Proof of Corollary \ref{b64}}
\label{c05}
To begin with observe that the
continuity of $f_0$, at least, on $\RR^n\setminus\{0\}$ implies~:
$$
|\dr_0|\;=\; -\dr_0\; \leq \; \min_{1\leq i\leq n}r_{0,i}\; \leq \; \max_{1\leq i\leq n}r_{0,i}
\; <\; d_{V_0}
\  .
$$
Then, let $V$ be the function given in Theorem \ref{b26}
and, since $\dr_0 < 0 < \dr_\infty$, the function
$\phi(x)=V(x)^\frac{d_{V_0}+\dr_0}{d_{V_0}} \;+\; V(x)^\frac{d_{V_\infty}+\dr_\infty}{d_{V_\infty}}$
is homogeneous in the bi-limit with weights $r_0$ and $r_\infty$, degrees $d_{V_0}+\dr_0$ and $d_{V_\infty}+\dr_\infty$ and approximating functions $V(x)^\frac{d_{V_0}+\dr_0}{d_{V_0}}$ and  $V(x)^\frac{d_{V_\infty}+\dr_\infty}{d_{V_\infty}}$.
Moreover, the function $\zeta(x)\,=\,-\frac{\partial V}{\partial x}(x)\,f(x)$ is homogeneous in the bi-limit with the same weights and degrees as $\phi$.
Furthermore, since the function $\zeta$ and its homogeneous approximations are positive definite,
Corollary \ref{b25} yields a strictly positive real number $c$
such that~:
\begin{equation}
\frac{\partial V}{\partial x}(x)\,f(x) \;\leq\;
-c\,\left(V(x)^\frac{d_{V_0}+\dr_0}{d_{V_0}} \;+\; V(x)^\frac{d_{V_\infty}+\dr_\infty}{d_{V_\infty}}\right)
\qquad \forall x\in\RR^n\ .
\end{equation}
Let $x_{ic}$ in $\RR^n\setminus\{0\}$ be the initial condition of a solution of the system $\dot x = f(x)$, and
$V_{x_{ic}}~:\RR_+\rightarrow\RR_+$ be the function of time given by the
evaluation of $V$ along this solution.
% It is non increasing. So there
% exsist a maximal interval $[0,T_{x_{ic}})$, $T_{x_{ic}}\leq +\infty $
% where it is non zero.
%startmodif
% Since~:
% $$
% a\;+\; b\geq \frac{1}{1/a + 1/b}
% \qquad \forall a>0\: ,\; b>0\  ,
% $$
% we have~:
% $$
% \frac{
% \dot {\overparen{V_{x_{ic}}(t)}}
% }{
% V_{x_{ic}}(t)^\frac{d_{V_0}+\dr_0}{d_{V_0}}}
% \;+\;
% \frac{
% \dot {\overparen{V_{x_{ic}}(t)}}
% }{
% V_{x_{ic}}(t)^\frac{d_{V_\infty}+\dr_\infty}{d_{V_\infty}}}
% \;\leq\; -2c
% \qquad \forall t\in [0,T_{x_{ic}})
% \  .
% $$
% and therefore
% $$
% \frac{d_{V_0}}{|\dr_0|}
% \left[
% V_{x_{ic}}(t)^\frac{|\dr_0|}{d_{V_0}}
% -
% V_{x_{ic}}(0)^\frac{|\dr_0|}{d_{V_0}}
% \right]
% \;+\;
% \frac{d_{V_\infty}}{\dr_\infty}\left[
%  \frac{1}{V_{x_{ic}}(0)^\frac{\dr_\infty}{d_{V_\infty}}}
% -\frac{1}{V_{x_{ic}}(t)^\frac{\dr_\infty}{d_{V_\infty}}}
% \right]
% \; \leq \; -2c t
% \qquad \forall t\in [0,T_{x_{ic}})
% $$
Then~:
$$
\dot {\overparen{V_{x_{ic}}(t)}}\;\leq\; -c\,
V_{x_{ic}}(t)^\frac{d_{V_\infty}+\dr_\infty}{d_{V_\infty}}
\qquad \forall t\geq 0
\  ,
$$
from which we get~:
$$
V_{x_{ic}}(t)
\; \leq \;
\frac{1}{
\left(
\frac{\dr_\infty }{d_{V_\infty}}\,  ct
+V(x_{ic})^{-\frac{\dr_\infty}{d_{V_\infty}}}
\right)^{\frac{d_{V_\infty}}{\dr_\infty}}
}
\; \leq \;
\frac{1}{
\left(
\frac{\dr_\infty }{d_{V_\infty}}\,  ct
\right)^{\frac{d_{V_\infty}}{\dr_\infty}}
}
\qquad \forall t> 0
\  .
$$
Therefore, setting $
T_1\;=\; \frac{d_{V_\infty}}{c\dr_\infty }
$,
we have~:
$$V_{x_{ic}}(t)\;\leq\;1\qquad \forall t\geq T_1
\quad ,\qquad
\forall x_{ic} \in \RR^n\ ,$$
and~:
$$
\dot {\overparen{V_{x_{ic}}(t)}}\;\leq\;
-c\,V_{x_{ic}}(t)^\frac{d_{V_0}-|\dr_0|}{d_{V_0}}
\qquad \forall t\geq 0
\  .
$$
As a result, we get~:
\begin{eqnarray*}
V_{x_{ic}}(t)& \leq&
\max\left\{
\left(-\frac{|\dr_0|}{d_{V_0}}\,  c(t-T_1)+V_{x_{ic}}(T_1)^\frac{|\dr_0|}{d_{V_0}}\right)^\frac{d_{V_0}}{|\dr_0|},0\right\}
\  ,
\\
&\leq&
\max\left\{\left(1-\frac{|\dr_0|}{d_{V_0}}\, c (t-T_1)\right)^\frac{d_{V_0}}{|\dr_0|},0\right\}
\qquad \forall t\geq T_1
\  .
\end{eqnarray*}
Therefore, setting
$
T_2\;=\;
\frac{d_{V_0}}{c|\dr_0|}
$,
yields~:
$$V_{x_{ic}}(t)\;=0\qquad \forall t\; \geq \;
T_1+T_2=
\frac{1}{c}\left(\frac{d_{V_\infty}}{\dr_\infty}+\frac{d_{V_0}}{|\dr_0|}\right)
\quad ,\qquad \forall x_{ic} \in \RR^n\ ,
$$
hence the claim.

\end{document}